\documentclass{amsart}

\usepackage{amscd}
\usepackage{amssymb}
\usepackage{amsmath}
\usepackage{amsfonts}
\usepackage{enumerate}
\usepackage{graphicx}
\usepackage{epic}
\usepackage[all,cmtip]{xy}
\usepackage{overpic}
\usepackage{caption}
\usepackage{tikz-cd}
\usepackage{setspace}

\newtheorem{thm}{Theorem}[section]
\newtheorem{prop}[thm]{Proposition}
\newtheorem{lem}[thm]{Lemma}
\newtheorem{cor}[thm]{Corollary}
\newtheorem{conj}[thm]{Conjecture}

\theoremstyle{definition}

\newtheorem{rem}[thm]{Remark}

\newtheorem*{que*}{Question}
\newtheorem*{AI}{AI disclosure}
\newtheorem*{nota}{Notations and conventions}

\newcommand{\kk}{\Bbbk}

\newcommand{\mm}{\mathbf{m}}

\newcommand{\yy}{\mathbf{y}}

\newcommand{\rr}{\mathbf{r}}

\newcommand{\Aa}{\mathcal {A}}

\newcommand{\JJ}{\mathcal {J}}

\newcommand{\OO}{\mathcal {O}}

\newcommand{\MM}{\mathcal {M}}

\newcommand{\II}{\mathcal {I}}
\newcommand{\EE}{\mathcal {E}}

\newcommand{\Zz}{\mathbb {Z}}
\newcommand{\Ff}{\mathbb {F}}
\newcommand{\Cc}{\mathbb {C}}

\newcommand{\Qq}{\mathbb {Q}}

\newcommand{\TT}{\mathcal {T}}

\DeclareMathOperator{\im}{Im}

\DeclareMathOperator{\Hom}{Hom}

\DeclareMathOperator{\Gal}{Gal}

\newcommand{\lh}{\lhook\joinrel\xrightarrow}
\newcommand{\w}{\widetilde}
\newcommand{\wh}{\widehat}
\newcommand{\ha}{\hookrightarrow}
\newcommand{\xr}{\xrightarrow}

\numberwithin{equation}{section}
\usepackage[colorlinks,linktocpage]{hyperref}
\hypersetup{citecolor=blue,linkcolor=blue}
\begin{document}
	\title[Counterexamples to a conjecture of Adams]{Counterexamples to a conjecture of Adams}
	\author[F.~Fan]{Feifei Fan}
	\thanks{This work was supported by the National Natural Science Foundation of China (Grant No. 12271183) and by the Guangdong Basic and Applied Basic Research Foundation (Grant No. 2023A1515012217).}
	\address{Feifei Fan, School of Mathematical Sciences, South China Normal University, Guangzhou, 510631, China.}
	\email{fanfeifei@mail.nankai.edu.cn}
	\subjclass[2020]{55R40, 55T10, 13A50}
	\keywords{classifying spaces, projective unitary groups, elementary abelian $p$-subgroups}
	\maketitle
	\begin{abstract}
	A conjecture due to J. F. Adams states that, for any odd prime $p$, the mod $p$ cohomology ring of the classifying space of a connected compact Lie group is detected by its elementary abelian $p$-subgroups. In this paper, we show that for any odd prime $p$ the mod $p$ cohomology ring of the classifying space of the projective unitary group $PU(p^2)$ is not detected by its elementary abelian $p$-subgroups, providing counterexamples to this conjecture. We also obtain several algebraic results as byproducts.
	\end{abstract}
	
	\section{Introduction}\label{sec:introduction}
	
	Let $BG$ be the classifying space of a topological group $G$. The  cohomology of $BG$ with coefficients in $\Zz$ or fields of prime characteristic  plays a fundamental role in algebraic topology, particularly in the theory of characteristic classes. 
	
	Let $G$ be a compact and connected Lie group, and let $p$ be a fixed prime. 
	One might expect that the mod $p$ cohomology classes of $BG$ can be detected by certain well-behaved subgroups of $G$, thereby facilitating the computations of $H^*(BG;\Ff_p)$ and of cohomology operations, such as Steenrod operations.
	We say that a class $\alpha\in H^*(BG;R)$ is detected by a subgroup $H\subset G$ if its image in $H^*(BH;R)$ is not zero under the restriction homomorphism $H^*(BG;R)\to H^*(BH;R)$. 
	
	By a classical result of Borel \cite{Borel53,Borel55}, if $H^*(BG;\Zz)$ has no $p$-torsion, then $H^*(BG;\Ff_p)$ is a polynomial algebra generated by even degree elements, and it can be detected by a maximal torus $T_G$ of $G$. That is, the inclusion $T_G\hookrightarrow G$ induces a monomorphism of cohomology rings
	\[H^*(BG;\Ff_p)\hookrightarrow H^*(BT_G;\Ff_p),\] 
	whose image is $H^*(BT_G;\Ff_p)^W$, the subring of invariants under the conjugation $W$-action with $W$ the Weyl group of $G$. 
	
	If $H^*(BG;\Zz)$ has $p$-torsion, the above restriction homomorphism to
	the maximal torus has a nontrivial kernel, since every class of odd degree restricts to zero in $H^*(BT_G;\Ff_p)$. We are then led to consider the restriction to elementary abelian $p$-subgroups. Suppose that $\EE_p(G)$ is a set of representatives of all
	conjugacy classes of maximal elementary abelian $p$-subgroups of $G$. By Quillen's result \cite[Theorem 6.2]{Qui71}, the kernel of the restriction map 
	\begin{equation}\label{eq:map}
	\Phi:H^*(BG;\Ff_p)\to \prod_{E\in\EE_p(G)}H^*(BE;\Ff_p)
    \end{equation} 
	contains only nilpotent elements for any compact connected Lie group $G$ and any prime $p$. For $p>2$, a stronger conjecture was made by J. F. Adams.
	\begin{conj}[Adams, see \cite{VV05}]\label{conj:Adams}
	Let $G$ be a compact connected Lie group. Then for any odd prime $p$, the map \eqref{eq:map} is injective.
	\end{conj}
	
	For a subgroup $E\subset G$, the \emph{Weyl group $W_G(E)$} is defined to be the quotient group $N_G(E)/C_G(E)$ of the normalizer $N_G(E)$ of $E$ by its centralizer $C_G(E)$. In particular, if $E\cong (\Zz/p)^k$, then $W_G(E)$ can be identified with a subgroup of $GL_k(\Ff_p)$ whose elements are induced by conjugation in $G$. Indeed, the restriction map \eqref{eq:map} has image in the subring of invariants $H^*(BE;\Ff_p)^{W_G(E)}$ for each product factor.
	
    For exceptional Lie groups, we need only consider cases $G=F_4,E_6,E_7,E_8$ with $p=3$, and $G=E_8$ with $p=5$, since in all other cases with $p$ odd, $H^*(BG;\Zz)$ has no $p$-torsion.
    The conjecture has been proved for $(F_4;p=3)$ \cite{Bro07}, and for $(E_6,E_7;p=3)$, $(E_8;p=5)$ \cite[Theorem 5.1]{KY08}.  
    
    Another family relevant to Conjecture \ref{conj:Adams} is that of the \emph{projective unitary groups}  $PU(n)$ obtained as the quotient group of the unitary group of degree $n$, $U(n)$, by its center $S^1=\{e^{i\theta}I_n:\theta\in [0, 2\pi]\}$. Vavpeti\v c and
    Viruel \cite{VV05} proved that Conjecture \ref{conj:Adams} holds for $G=PU(p)$ with $p$ an arbitrary odd prime (the special case $G=PU(3)$ was also proved by Kono and Yagita in an earlier work \cite{KY93}). 
    
    In this paper, we show that Conjecture \ref{conj:Adams} is not true for $PU(p^2)$. 
	 
	\begin{thm}\label{thm:main}
		Let $p$ be an odd prime and let $\chi$ be a generator of \[H^3(BPU(p^2);\Ff_p)\cong\Ff_p.\] 
		Then there exists a class $\alpha\in H^{2p^2(p-1)}(BPU(p^2);\Ff_p)$ such that $\chi\alpha\neq 0$ and $\Phi(\chi\alpha)=0$, where  $\Phi$ is the restriction map in \eqref{eq:map}.
	\end{thm}
	In fact, the classes $\chi$ and $\alpha$ in Theorem \ref{thm:main} are both mod $p$ reductions of integral classes. 
	The construction of these counterexamples is based on some calculations of the Serre spectral sequence of the fibration:
	\[
		BU(n)\to BPU(n)\to K(\Zz,3)
	\]
	for $n=p^2$, together with the invariant theory for actions of finite symplectic groups on polynomial rings. Here we propose a conjecture for general $n$ with $p^2|n$.
	\begin{conj}
	For any odd prime $p$ and any integer $n>0$ with $p^2|n$, the map $\Phi$ in \eqref{eq:map} is not injective for $G=PU(n)$.
	\end{conj}
	
	\subsection*{Organization of the paper} 
	Section \ref{sec:SS} collects preliminary results concerning the above Serre spectral sequence. In Section \ref{sec:BPU(p)}, we recall formulas  for the integral and $\Ff_p$-cohomology of $BPU(p)$ from \cite{Vis07} and \cite{Fan26b}, and study some properties of the spectral sequence of the fibration $BPU(p)\to K(\Zz,3)$, together with algebraic applications. The construction of the class $\alpha$ in Theorem \ref{thm:main} relies on several algebraic facts about symmetric polynomials, which are established in Sections \ref{sec:K_n}-\ref{sec:Delta for K_p^2}. 
	In Section \ref{sec:p-subgroups}, the maximal nontoral elementary abelian $p$-subgroups of $PU(n)$ and their Weyl groups are described. Section \ref{sec:symplectic invariants} is a review of the invariant theory of finite symplectic groups $Sp_{2m}(\Ff_p)$. Here $Sp_{2m}(\Ff_p)$ is the Weyl group of a maximal nontoral elementary abelian $p$-subgroup of $PU(p^m)$.  In Section \ref{sec:image of Phi_1}, we study the restriction map for another nontoral elementary abelian $p$-subgroup of $PU(p^2)$. After all these preliminaries, we give the proof of Theorem \ref{thm:main} in Section \ref{sec:proof}.
	Finally, in Section \ref{sec:detect} we show that the integral class $\chi\alpha$ is detected by an abelian subgroup $(\Zz/p^2)^2\subset PU(p^2)$. 
	
	\begin{nota}
		Given a topological space $X$, we use the simplified notation $H^*(X)$ to denote the integral cohomology $H^*(X;\Zz)$.
		For an abelian group $A$ and a prime $p$, let $A_{(p)}$ denote the localization of $A$ at $p$. For a fixed odd prime $p$, let $P^k$ be the $k$th Steenrod power operation and let $\beta$ be the Bockstein in the mod $p$ Steenrod algebra. For a group homomorphism $\pi:G\to G'$, $B\pi:BG\to BG'$ denotes the corresponding map on classifying spaces. 
		
		For the $n$-dimensional torus $T^n$, the cohomology $H^*(BT^n)\cong \Zz[v_1,\dots,v_n]$ is a polynomial algebra in $n$ variables. We use the convention that the cohomology degree and the polynomial degree of each generator $v_i$ are $2$ and $1$, respectively. Similarly, if $A=(\Zz/p)^n$ is an elementary abelian $p$-group for an odd prime $p$, the polynomial part of $H^*(BA;\Ff_p)$ is $\Ff_p[v_1,\dots,v_n]$, and we use the same convention for the degrees of each generator.
	\end{nota}
	
	\begin{AI}
		In an earlier version of this paper, Theorem \ref{thm:main} was proved only for $p=3$. Later, the author used AI to generalize  Proposition \ref{prop:e} to arbitrary odd primes, and to assist in proving Proposition \ref{prop:L_p}, thereby establishing Theorem \ref{thm:main} in its present generality.
		The AI-generated proofs were verified and rewritten by the author, who takes full responsibility for the contents of the article.
	\end{AI}

\section{A spectral sequence converging to $H^*(BPU(n))$}\label{sec:SS}
    The short exact sequence of Lie groups
	\[1\to S^1\to U(n)\to PU(n)\to 1\]
	induces  a fibration of their classifying spaces 
	\[BS^1\to BU(n)\to BPU(n).\]
	Note that $BS^1$ has the homotopy type of the Eilenberg-Mac Lane space $K(\Zz,2)$. So we obtain another induced fibration:
	\begin{equation}\label{eq:fibration}
	BU(n)\to BPU(n)\xr{\chi} K(\Zz,3).
	\end{equation}
    Let $x$ be a generator of $H^3(K(\Zz,3))\cong \Zz$. It is known that $H^i(BPU(n))=0$ for $i=1,2$, and that $H^3(BPU(n))\cong\Zz/n$ is generated by $\chi^*(x)$. For simplicity, we use $\chi$ to denote $\chi^*(x)$. 
    
    Let $E_*^{*,*}(n)$ (or simply $E_*^{*,*}$) be the cohomology Serre spectral sequence  with integer coefficients  associated to \eqref{eq:fibration}. It is of the form
    \begin{equation}\label{eq:Serre}
    	E_2^{s,t}=H^s(K(\Zz,3);H^t(BU(n)))\Longrightarrow H^{s+t}(BPU(n)).
    \end{equation}
    Since $H^*(BU(n))$ is concentrated in even degrees, all $d_2$ differentials are trivial, and then $E_2^{*,*}=E_3^{*,*}$.
    In order to compute this spectral sequence, we first recall the cohomology of $K(\Zz,3)$.
	 \subsection{The cohomology of $K(\Zz,3)$}\label{subsec:K(Z,3)}
	 The integral cohomology of $K(\Zz,3)$ is described in detail by Gu \cite{Gu21}, using Cartan's method \cite{Car54}. Here, we only consider the $p$-local cohomology for any odd prime $p$. Note that $H^*(-;\Zz_{(p)})\cong H^*(-)_{(p)}$.
	 \begin{prop}[{\cite[Proposition 2.1]{Gu21b}}]\label{prop:K(Z,3)}
	 	The graded ring $H^*(K(Z,3))_{(p)}$ is generated by $x\in H^3(K(Z,3))_{(p)}\cong \Zz_{(p)}$ and $p$-torsion elements $y_{p,I}$, where the index $I=(i_1,\dots,i_m)$ ranges over all ordered sequences of nonnegative integers $i_1<\cdots<i_m$, and the degree of $y_{p,I}$ is 
	 	\[|y_{p,I}|=1+\sum_{j=1}^m(2p^{i_j+1}+1).\]
	 \end{prop}
	 For simplicity, we write $y_{p,i}:=y_{p,(i)}$. We also use $y_{p,I}$ to denote its image in $H^*(BPU(n))$ under the induced homomorphism $H^*(K(\Zz,3))\to H^*(BPU(n))$ whenever there is no ambiguity.
	 
	 Let $x,y_{p,I}$ also denote their mod $p$ reductions in $H^*(K(\Zz,3);\Ff_p)$, and define
	 \[x_{p,k}:=P^{p^k}P^{p^{k-1}}\cdots P^1(x),\ \ k=0,1,2,\dots\]
	 Let $\beta_\Zz:H^*(-;\Ff_p)\to H^*(-)$ denote the Bockstein homomorphism associated to the short exact sequence
	 \[0\to\Zz\xr{\cdot p}\Zz\to\Zz/p\to0.\]
	 Then the element $y_{p,I}\in H^*(K(\Zz,3))$ can be expressed as
	 \begin{equation}\label{eq:y_I}
	 	y_{p,I}=\beta_\Zz(x_{p,i_1}\cdots x_{p,i_m})\ \text{ for }\  I=(i_1,\dots,i_m).
	 \end{equation}
	 	
	\subsection{The differential $d_3$ in $E_*^{*,*}$}\label{subsec:d_3}
Recall that the cohomology ring of  $BU(n)$ is
\[H^*(BU(n))=\Zz[c_1,\dots,c_n],\ |c_i|=2i.\]
Here $c_i$ is the $i$th universal Chern class of the classifying space $BU(n)$ for $n$-dimensional complex bundles. For a maximal torus $T^n$ of $U(n)$,  let
 $\psi:T^n\to U(n)$ be the inclusion.
Then we have 
\[B\psi^*(c_i)=\sigma_i\in H^*(BT^n)\cong \Zz[v_1,\dots,v_n],\] where $\sigma_i$ is the $i$th elementary symmetric polynomial in the variables $v_1,\dots,v_n$.
To simplify notation, for the rest of this paper, we use $R[n]$ to denote the polynomial ring $R[v_1,\dots,v_n]$ with coefficient ring $R$.

For an integer $n>0$, let $\Lambda_n$ be the ring of symmetric polynomials over $\Zz$ in $n$ variables. It is known that $\Lambda_n=\Zz[\sigma_1,\dots,\sigma_n]$.  

Let $\nabla=\sum_{i=1}^n\partial/\partial v_i$ be the linear differential operator acting on the polynomial ring $\Zz[v_1,\dots,v_n]$.
$\nabla$ clearly preserves symmetric polynomials, so $\nabla$ can be restricted to $\Lambda_n$. For $\sigma_k\in \Lambda_n$ (with the convention that $\sigma_0=1$), one checks that
\begin{equation}\label{eq:nabla}
\nabla(\sigma_k)=(n-k+1)\sigma_{k-1}.	
\end{equation}

The differential $d_3$ of the $E_3$-page of the spectral sequence \eqref{eq:Serre} can be described as follows:
\begin{prop}[{\cite[Corollary 3.10]{Gu21}}]\label{prop:d_3}
	For $f\in E_3^{0,*}\cong H^*(BU(n))\cong \Lambda_n$, $\xi\in E_3^{*,0}\cong H^*(K(\Zz,3))$, we have
\[
		d_3(f\xi)=\nabla(f)x\xi.
\]
\end{prop}

Let $K_n$ be the kernel of the map $\nabla:\Lambda_n\to\Lambda_n$.
The following result of Crowley and Gu gives a nice description of the cohomology ring of $BPU(n)$ modulo torsion elements.
\begin{thm}[{\cite[Theorem 1.3]{CG26}}]\label{thm:E_4}
	Let $\pi:U(n)\to PU(n)$ be the quotient map, and let 
	$W$ be the Weyl group of $PU(n)$. Then we have 
	\[K_n\cong E_4^{0,*}=E_\infty^{0,*}\cong \im (B\pi)^*\cong H^*(BT_{PU(n)})^W\cong H^*(BPU_n)/torsion.\]
\end{thm}
\begin{rem}\label{rem:Weyl gorup invariant}
	In fact, the restriction map
	\[H^*(BPU_n;R)\to H^*(BT_{PU(n)};R)^W\] 
	is onto for any coefficient ring $R$. This follows from  the proof of \cite[Theorem 1.1]{CG26}.  
	More precisely, the map $H^*(BN;R)\to H^*(BT_{PU(n)};R)^W$ is onto for the normalizer $N$ of $T_{PU(n)}$ by \cite[Theorem 1.8]{CG26} and \cite[(3.31)]{CG26}, and then \cite[Theorem I.12]{Fes81} gives the result.
\end{rem}

\section{On the cohomology ring of $BPU(p)$}\label{sec:BPU(p)}
In this section, we assume that $p$ is an odd prime. In \cite{Vis07}, Vistoli provides a description of the integral cohomology ring of $BPU(p)$ (see also \cite{Fan26b}). 

We denote by $T^p$ the standard maximal torus in $U(p)$ consisting of diagonal
matrices, and by $A_p$ the algebraic
group of $p$th roots of $1$ over $\Cc$. Consider the embedding 
\[A_p\hookrightarrow T^p,\ \ \omega\mapsto (\omega,\omega^2,\dots,\omega^{p-1},1),\ \omega=e^{2\pi i/p}.\]
This map induces a  homomorphism  
\[\Theta_p:H^*(BT^p)\cong \Zz[v_1,\dots,v_p]\to\Zz[\eta]/\langle p\eta\rangle\cong H^*(BA_p),\ v_i\mapsto i\eta.\] 
Let $\overline\Theta_p$ be the composition of $\Theta_p$ with the mod $p$ reduction, and let $I_p$ be the kernel of the restriction of $\overline\Theta_p$ to $K_p\subset H^*(BT^p)$. 

\begin{thm}[{\cite[Theorem 3.4]{Vis07}}]\label{thm:vistoli}
	For any prime $p>2$, the integral cohomology ring of $BPU(p)$ is given by
	\[ H^*(BPU(p))\cong\frac{K_p\otimes \Zz[\chi,\gamma]}{\langle\chi^2,\,p\chi,\,p\gamma,\,\chi I_p,\,\gamma I_p\rangle},\]
	where $\chi=\chi^*(x)$ and $\gamma=\chi^*(y_{p,0})$ for the map $\chi:BPU(p)\to K(\Zz,3)$.
\end{thm}
The author \cite{Fan26b} provides a similar formula for the mod $p$ cohomology of $BPU(p)$. We follow the notation of Theorem \ref{thm:vistoli} and set 
\[\overline K_p=\im(K_p\to\Lambda_p\otimes\Ff_p),\ \ \overline I_p=\im(I_p\to\Lambda_p\otimes\Ff_p).\] 
\begin{thm}[{\cite[Theorem 1.3]{Fan26b}}]\label{thm:mod p BPU(p)}
	For any prime $p>2$, let $L_p$ be the subring of $\Lambda_p\otimes\Ff_p$ generated by $\overline K_p$ and $\sigma_1$. Then there is a graded isomorphism
	\[ H^*(BPU(p);\Ff_p)\cong \frac{L_p\otimes \Ff_p[\gamma]\otimes\Lambda_{\Ff_p}[\chi,\zeta]}{\langle \sigma_1\chi,\,\sigma_1\zeta,\,\sigma_1\gamma+\chi\zeta,\,\chi\overline I_p,\,\zeta\overline I_p,\,\gamma\overline I_p\rangle},\]
	where $\chi,\gamma$ are as in Theorem \ref{thm:vistoli} and $\zeta=\chi^*(x_{p,0})$.
\end{thm}

For the element $\delta=\prod_{i\neq j}(v_i-v_j)\in K_p$, one easily checks that 
\begin{equation}\label{eq:Theta(delta)}
	\Theta_p(\delta)=-\eta^{p^2-p}\neq 0.
\end{equation}
It is shown in \cite[Proposition 3.1]{Vis07} that the image of $\Theta_p|_{K_p}$ is the subring generated by $\Theta_p(\delta)$. Hence, as a $\Zz$-module, 
\begin{equation}\label{eq:module structure}
	H^*(BPU(p))\cong K_p\oplus \left((\Lambda_{\Ff_p}[\chi]\otimes\Ff_p[\gamma])^+\otimes\Ff_p[\delta]\right).
\end{equation}

We also consider the integral cohomology Serre spectral sequence $E_*^{*,*}$ of the fibration $BPU(p)\to K(\Zz,3)$. The Serre spectral sequence gives, in each total degree $n$, a filtration
\[0\lh{E_\infty^{n,0}}F_n^n\lh{E_\infty^{n-1,1}}F_{n-1}^n\ha\cdots\ha F_1^n\lh{E_\infty^{0,n}}F_0^n=H^n(BPU(p)),\]
where $F_{k+1}^n\lh{E_\infty^{k,n-k}}F_k^n$ means that there is a short exact sequence
\[0\to F_{k+1}^n\to F_k^n\to E_\infty^{k,n-k}\to 0.\]
Theorem \ref{thm:E_4} shows that  the kernel of the surjection $H^*(BPU(p))\to E_\infty^{0,*}$ is the subgroup $\TT$ consisting of torsion elements. 
Since the first nontrivial column of $E_*^{>0,*}$ is $E_*^{3,*}$, there is a surjection $\phi:\TT\to E_\infty^{3,*}$. Note that \eqref{eq:module structure} gives a decomposition 
\[\TT=(\chi\Ff_p[\delta])\oplus(\Ff_p[\gamma]^+\otimes \Ff_p[\delta])\oplus(\chi\Ff_p[\gamma]^+\otimes\Ff_p[\delta]).\]
\begin{prop}\label{prop:the map phi}
	The surjection $\phi:\TT\to E_\infty^{3,*}$ is given by
	\[\phi(\chi\delta^k)=x\delta^k,\ \ \phi(\gamma^i\delta^k)=\phi(\chi\gamma^i\delta^k)=0,\ \ i>0,\,k\geq 0.\]
	Furthermore, when restricted to $\chi\Ff_p[\delta]$, $\phi$ is an isomorphism.
\end{prop}
\begin{proof}
	Since $\chi=\chi^*(x)$, $\gamma=\chi^*(y_{p,0})$ and $\deg(y_{p,0})>3$, the desired formula for $\phi$ follows. Since $d_3=\nabla$ is the only possible differential mapping to $E_{i}^{3,*}$ ($i\geq 2$) and $K_p$ consists of permanent cycles by Theorem \ref{thm:E_4}, it follows that 
	\[K_p/\left(K_p\cap \nabla(\Lambda_p)\right)\subset E_\infty^{3,*}.\]
	Thus, the following lemma shows that $\phi$ restricted to $\chi\Ff_p[\delta]$ is injective, hence an isomorphism because $\phi$ is surjective.
\end{proof}

\begin{lem}\label{lem:delta}
$p\Zz[\delta]\subset\nabla(\Lambda_p)$ and $\delta^k\notin\nabla(\Lambda_p)$ for any integer $k\geq 0$.	
\end{lem}
\begin{proof}
	Since $p\delta^k=\nabla(\sigma_1\delta^k)$ by \eqref{eq:nabla}, $p\Zz[\delta]\subset\nabla(\Lambda_p)$.
	 If $\nabla(f)=\delta^k$ for some $f\in\Lambda_p$, then $\sigma_1\delta^k-pf\in K_p$, and then Theorem \ref{thm:mod p BPU(p)} shows that 
	\[\sigma_1\delta^k\in \overline K_p\subset L_p\subset H^*(BPU(p);\Ff_p).\]
	Since $\sigma_1\delta^k$ is the mod $p$ reduction of an integral class, $\beta(\sigma_1\delta^k)=0$. But this contradicts the fact that the image of $\beta(\sigma_1\delta^k)$ in the mod $p$ cohomology of the classifying space of a subgroup $\Gamma$ of $PU(p)$ is nonzero by Proposition \ref{prop:Mui} and \eqref{eq:f,h} below.
\end{proof}

 Note that $K_n\cap\nabla(\Lambda_n)$ is an ideal of $K_n$ since if $a\in K_n\cap \nabla(\Lambda_n)$, say $a=\nabla(f)$ for some $f\in\Lambda_n$, and $b\in K_n$, then $\nabla(fb)=ab$. So we have an algebraic application of Proposition \ref{prop:the map phi} and Lemma \ref{lem:delta}.
\begin{cor}\label{cor:delta}
	The quotient ring $K_p/\left(K_p\cap\nabla(\Lambda_p)\right)$ is isomorphic to $\Ff_p[\delta]$.
\end{cor}
We also have another description of the ideal $I_p\subset K_p$.
\begin{cor}\label{cor:I_p}
	$K_p\cap\nabla(\Lambda_p)=I_p$.
\end{cor}
\begin{proof}
By \cite[Proposition 3.1]{Vis07} (see also \cite[Corollary 2.6]{Fan26b}), $K_p/I_p\cong \Ff_p[\delta]$. Hence by Corollary \ref{cor:delta}, it suffices to show that $I_p\subset \nabla(\Lambda_p)$. 
Given $f\in I_p$, we can  write it as $f=f_0+f_1$, where $f_0\in\Zz[\delta]$, $f_1\in K_p\cap\nabla(\Lambda_p)$, by Corollary \ref{cor:delta}. Once we prove that $f_0\in p\Zz[\delta]$, it will follow that $f\in\nabla(\Lambda_p)$ since $p\Zz[\delta]\subset\nabla(\Lambda_p)$ by Lemma \ref{lem:delta}.

Since $f\in K_p$, there is a class $\rho\in H^*(BPU(p))$ such that $(B\pi)^*(\rho)=f$ by Theorem \ref{thm:E_4}, where $\pi:U(p)\to PU(p)$ is the quotient map. Thus $\rho$ lies in the kernel of the composition of $(B\pi)^*$ and $\overline\Theta_p$, which is the ideal $\langle I_p,\chi,\gamma\rangle$ in the formula of Theorem \ref{thm:vistoli}. Therefore, $\chi\rho\in\langle\chi\gamma\rangle$ since $\chi^2=\chi I_p=0$. It follows that $\chi\rho=0$ in $E_\infty^{3,*}$ since $\chi\gamma\in E_\infty^{2p+5,0}$. Note that the representative of $\chi\rho$ in $E_\infty^{3,*}$ is $xf=xf_0$.
Then Proposition \ref{prop:the map phi} implies that $f_0\in p\Zz[\delta]$. 
\end{proof}

\section{On the ring $K_n$}\label{sec:K_n}
In this section, we present some purely algebraic results on the subring $K_n$ of $\Lambda_n$, which can be used to study the cohomology of $BPU(n)$. Recall from Theorem \ref{thm:E_4} that $K_n$ is the quotient ring of $H^*(BPU(n))$ by torsion elements.

For any integer $n>0$, define the ring homomorphism
\begin{equation}\label{eq:theta}
		\Theta:\Zz[n]\to\Zz, \ \ \Theta(v_i)=i.
\end{equation}

\begin{lem}\label{lem:Theta mod p}
	Let $p$ be a prime. Then for $\sigma_k\in\Lambda_p$, $1\leq k\leq p$, we have
	\[\Theta(\sigma_k)\equiv\begin{cases}
		-1\bmod p,&k=0,p-1,\\
		0\bmod p,&\text{otherwise.}
	\end{cases}\]  	
\end{lem}
\begin{proof}
	This comes from the fact that $x^{p-1}-1=\prod_{i=1}^{p-1}(x-i)$ in $\Ff_p[x]$.
\end{proof}
Suppose that $p|n$ for some prime $p$. Define a ring homomorphism  
\begin{equation}\label{eq:the map Delta}
\Delta:\Zz[n]\to \Zz[p],\ \ \Delta(v_i)=v_{r(i)},	
\end{equation}
where $r(i)\in[p]$ and $i\equiv r(i)$ mod $p$. $\Delta$ clearly commutes with the differential operator $\nabla$.
Since 
\[\Delta(\prod_{i=1}^n(1+v_i))=(\prod_{i=1}^p(1+v_i))^{\frac{n}{p}},\]
the restriction of $\Delta$ to $\Lambda_n$ is given by
\begin{equation}\label{eq:diagonal map}
	\Delta(1+\sigma_1+\cdots+\sigma_n)=(1+\sigma_1+\cdots+\sigma_p)^{\frac{n}{p}}.
\end{equation}

\begin{lem}\label{lem:Theta for Lambda_p^r}
	Let $p$ be a prime. Then for $\sigma_k\in\Lambda_{p^r}$ ($r>1$), $1\leq k\leq p^r$, we have
	\[\Theta(\sigma_k)\equiv\begin{cases}
		-1\bmod p,&k=p^{r-1}(p-1),\\
		0\bmod p,&\text{otherwise.}
	\end{cases}\]
\end{lem}
\begin{proof}
	Since $\Theta\Delta\equiv \Theta$ mod $p$, the result follows from Lemma \ref{lem:Theta mod p} and \eqref{eq:diagonal map}.
\end{proof}

For any prime divisor $p$ of $n$, let $K_{n,p}$ be the kernel of the map \[\nabla\otimes\Ff_p:\Lambda_n\otimes\Ff_p\to \Lambda_n\otimes\Ff_p.\]
It is clear that $\overline K_n\subset K_{n,p}$, where $\overline K_n=\im(K_n\to\Lambda_n\otimes\Ff_p)$.
\begin{lem}\label{lem:K_n,p}
	Suppose 
	that $f=\sum_il_i\mm_i$, $0\neq l_i\in\Ff_p$, is a nonzero element of $K_{n,p}$, where $\mm_i$ are monomials in $\Ff_p[\sigma_1,\dots,\sigma_n]$. If \[\mm_i=\sigma_{i_1p}^{j_1}\sigma_{i_2p}^{j_2}\cdots\sigma_{i_sp}^{j_s},\  0<i_1<\cdots<i_s,\ j_1,\dots,j_s>0,\] 
	for some $i$, then $p$ divides $j_k$ for $1\leq k\leq s$.
\end{lem}
\begin{proof}
	By \eqref{eq:nabla}, we have \[\nabla(\mm_i)=\sum_{k=1}^sj_k(n-i_kp+1)\sigma_{i_1p}^{j_1}\cdots\sigma_{i_kp-1}\sigma_{i_kp}^{j_k-1}\cdots\sigma_{i_sp}^{j_s}.\]
	Moreover, it is easy to see from \eqref{eq:nabla} that if $\mm\in\Lambda_n$ is a monic monomial, then a nonzero multiple of the monomial $\sigma_{i_1p}^{j_1}\cdots\sigma_{i_kp-1}\sigma_{i_kp}^{j_k-1}\cdots\sigma_{i_sp}^{j_s}$ appears in the mod $p$ reduction of $\nabla(\mm)$ only if $\mm=\mm_i$. This implies that $p|j_k(n-i_kp+1)$ since $f\in K_{n,p}$. However, $p\nmid(n-i_kp+1)$ since $p|n$, so $p|j_k$. 
\end{proof}

\begin{lem}\label{lem:K_p^r}
Let $p$ be a prime. Then for any $r>1$ and any $f\in K_{p^r,p}$, we have $\Delta(f)\in\Ff_p[\sigma_1^{p^r},\dots,\sigma_p^{p^r}]$ and the coefficient of a monomial $\prod_{k=1}^s\sigma_{i_k}^{j_kp^r}$ in $\Delta(f)$ is equal to the coefficient of $\prod_{k=1}^s\sigma_{i_kp^{r-1}}^{j_kp}$ in $f$.
\end{lem}
\begin{proof}
	This is a consequence of \eqref{eq:diagonal map} and Lemma \ref{lem:K_n,p}.
\end{proof}
We end this section by constructing a spectral sequence map $E_*^{*,*}(n)\to E_*^{*,*}(p)$ for any integer $n>0$ with a prime divisor $p$. 

Consider the diagonal map of matrices 
\[U(p)\to U(n),\quad A\mapsto \begin{bmatrix}
	A&0&\cdots&0\\
	0&A&\cdots&0\\
	\vdots&\vdots&\ddots&\vdots\\
	0&0&\cdots&A
\end{bmatrix},\] which passes to $PU(p)\to PU(n)$.
It induces maps 
\begin{equation}\label{eq:bar Delta}
	\Delta:BU(p)\to BU(n)\ \text{ and }\ \bar\Delta:BPU(p)\to BPU(n),
\end{equation} 
and a commutative diagram of fibrations
\begin{equation}\label{eq:diag}
\begin{gathered}\xymatrix{
	BU(p)\ar[r]\ar^{\Delta}[d]&BPU(p)\ar[r]\ar[d]^{\bar\Delta}&K(\Zz,3)\ar[d]^{=}\\
	BU(n)\ar[r]&BPU(n)\ar[r]&K(\Zz,3)}
\end{gathered}
\end{equation}
Then the spectral sequence map $\bar\Delta^*:E_*^{*,*}(n)\to E_*^{*,*}(p)$ on $E_2$-pages is given by the induced homomorphism   $\Delta^*:H^*(BU(n))\to H^*(BU(p))$, which can be described as follows.

Let $c_i$ and $c_i'$ be the $i$th universal Chern classes of $BU(n)$ and $BU(p)$ respectively. 
Since the map $\Delta$ factors through the diagonal map $BU(p)\to (BU(p))^{\frac{n}{p}}$, the Whitney sum formula and the functorial property of total Chern classes give
\begin{equation}\label{eq:chern class}
	\Delta^*(1+c_1+\cdots+c_n)=(1+c_1'+\cdots+c_p')^{\frac{n}{p}}.
\end{equation}
Hence, $\Delta^*$ can be identified with the map $\Delta:\Lambda_n\to\Lambda_p$ defined above.

\section{An element $e\in K_{p^2}$}\label{sec:element e}
The main goal of this section is to construct an element $e\in K_{p^2}$, which will be used to define the class $\alpha\in H^*(BPU(p^2))$. Throughout this section, $p$ is an odd prime and  $\Theta$, $\Delta$ are the maps defined in the last section. The following proposition is the main result of this section.
\begin{prop}\label{prop:e}
	There exists an element $e\in K_{p^2}^{p^2(p-1)}$ such that 
	\[\Delta(e)\equiv 0\bmod p\ \text{ and }\ v_p(\Theta(e))=1,\]
	where $v_p(k)$ is the $p$-adic valuation of an integer $k$.
\end{prop}
Here we should mention that in an earlier version, this proposition was only proved for $p=3$. The general construction of the element $e$ was given by AI in a different way. Here we only introduce this AI-generated construction. 

Let $E(t)\in\Lambda_{p^2}[t]$ be the element $E(t)=\sum_{i=0}^{p^2}\sigma_it^i$.
Let $n=p^2(p-1)$ and define $e_i\in\Lambda_{p^2}$, $0\leq i\leq n$, by
\[E(t)^{p-1}=\sum_{i=0}^ne_i t^i.\]

\begin{lem}\label{lem:nabla e_i}
	$\nabla(e_i)=(n-i+1)e_{i-1}$ for $1\leq i\leq n$.
\end{lem}
\begin{proof}
	For any $f(t)=\sum_i f_i t^i\in \Lambda_{p^2}[t]$ with $f_i\in\Lambda_{p^2}$, write $\nabla f(t)=\sum_i\nabla(f_i)t^i$ and $f'(t)=\frac{d}{dt}f(t)$. By \eqref{eq:nabla}, one easily checks that
	\[\nabla E(t)=p^2tE(t)-t^2E'(t)=\sum_{i=0}^{p^2-1}(p^2-i)\sigma_i t^{i+1}.\]
	Then, using the Leibniz rule we have
	\[\begin{split}\nabla(E(t)^{p-1})&=(p-1)E(t)^{p-2}\nabla E(t)\\
		&=(p-1)E(t)^{p-2}\left(p^2tE(t)-t^2E'(t)\right)\\
		&=p^2(p-1)tE(t)^{p-1}-(p-1)t^2E(t)^{p-2}E'(t).\end{split}\]
	Since $n=p^2(p-1)$ and $\frac{d}{dt}E(t)^{p-1}=(p-1)E(t)^{p-2}E'(t)$, this becomes 
	\begin{equation}\label{eq:differential}
		\nabla(E(t)^{p-1})=ntE(t)^{p-1}-t^2\frac{d}{dt}E(t)^{p-1}.
	\end{equation}
	Substituting $E(t)^{p-1}=\sum_{i=0}^ne_i t^i$ into the right-hand side of \eqref{eq:differential} gives
	\[ntE(t)^{p-1}=\sum_{i=0}^n ne_it^{i+1}\ \text{ and }\ t^2\frac{d}{dt}E(t)^{p-1}=\sum_{i=0}^n ie_it^{i+1}.\]
	Therefore, the coefficient of $t^i$, $1\leq i\leq n$, in \eqref{eq:differential} is $(n-i+1)e_{i-1}$, which gives the desired formula for $\nabla(e_i)$.
\end{proof}
Consider the rational polynomial $\EE\in\Lambda_{p^2}\otimes\Qq$
\[\EE=\sum_{i=0}^n a_ie_ie_{n-i},\ \ a_i=\frac{(-1)^i}{\binom{n}{i}}\in\Qq.\]
\begin{lem}
	$\EE\in K_{p^2}\otimes\Qq$. 
\end{lem}
\begin{proof}
By Lemma \ref{lem:nabla e_i},
\[\nabla(\EE)=\sum_{i=1}^na_i(n-i+1)e_{i-1}e_{n-i}+\sum_{i=0}^{n-1}a_i(i+1)e_ie_{n-i-1}.\] 
Since $n$ is even, $a_i=a_{n-i}$, so the coefficient of $e_ie_{n-i-1}$, $0\leq i<n/2$, in $\nabla(\EE)$ is 
\[(i+1)(a_i+a_{n-i})+(n-i)(a_{i+1}+a_{n-i-1})=2\left((i+1)a_i+(n-i)a_{i+1}\right).\]
 But $a_{i+1}=-a_i(i+1)/(n-i)$, so this coefficient is zero. Hence $\EE\in K_{p^2}\otimes\Qq$.	
\end{proof}

\begin{proof}[Proof of Proposition \ref{prop:e}]
Since $n<p^3$, Kummer’s theorem shows that the $p$-adic valuation of $\binom{n}{i}$ is at most $2$, hence $p^2\EE\in K_{p^2}\otimes\Zz_{(p)}$. Choose an integer $q$ such that $p\nmid q$ and $qp^2\EE\in K_{p^2}$. We shall verify the two formulas of the proposition for the element $e:=qp^2\EE$.

Let $m=p^2-p$. The coefficient of $\sigma_m^k$ in $e_{km}$ is $\binom{p-1}{k}$ for $1\leq k\leq p-1$ and zero for $k=p$. Thus, the coefficient $A$ of $\sigma_m^p$ in $e$ is
\[A=qp^2\sum_{k=1}^{p-1}\frac{\binom{p-1}{k}\binom{p-1}{k-1}}{\binom{n}{km}}.\]
For $1\leq k\leq p-1$, $v_p(\binom{n}{km})=1$. Moreover, since $\frac{1}{p}\binom{n}{km}=\frac{1}{k}\binom{n-1}{km-1}$ and the $p$-adic expansions of $n-1$ and $km-1$ are
\begin{gather*}
	n-1=(p-2)p^2+(p-1)p+(p-1),\\
	km-1=(k-1)p^2+(p-k-1)p+(p-1),
\end{gather*}
Lucas’ theorem (see \cite[Lemma 3C.6]{Hat02}) therefore gives
\[\frac{1}{p}\binom{n}{km}\equiv\frac{1}{k}\binom{p-2}{k-1}\binom{p-1}{k}\binom{p-1}{p-1}\equiv (-1)^{k-1}(-1)^k=-1\bmod p.\]
Also,
\[\binom{p-1}{k}\binom{p-1}{k-1}\equiv(-1)^k(-1)^{k-1}=-1\bmod p.\]
Hence, $A/p\equiv -q$ mod $p$, and so $v_p(A)=1$.

By Lemma \ref{lem:K_p^r}, $\Delta(e)\not\equiv 0$ mod $p$ if and only if the mod $p$ reduction of $e$ contains a nonzero multiple of a monomial of the form $\prod_k\sigma_{i_kp}^{j_kp}$. On the other hand, every $e_i$ is a polynomial in $\sigma_j$'s with each monomial containing at most $p-1$ nontrivial factors $\sigma_j$, so that every monomial in $e$ contains at most $2(p-1)$ nontrivial factors $\sigma_j$. Thus, if $\prod_k\sigma_{i_kp}^{j_kp}$ is a monomial in $e$, then $p\sum_k j_k\leq 2(p-1)$, and so $\sum_k j_k=1$. This implies that the only possibility  is $\sigma_m^p$. However, we have already shown that the coefficient $A$ of $\sigma_m^p$ in $e$ is zero mod $p$, so $\Delta(e)\equiv 0$ mod $p$.

Note also that if  $\mm=\prod_k\sigma_{j_k}\in\Lambda_{p^2}^n$ is a monomial  other than a multiple of $\sigma_m^p$, then  $j_k\neq m$ for at least two indices $k$, so $v_p(\Theta(\mm))\geq 2$  by Lemma \ref{lem:Theta for Lambda_p^r}.
Since $v_p(A)=1$, using Lemma \ref{lem:Theta for Lambda_p^r} again we obtain the  formula $v_p(\Theta(e))=1$.
\end{proof}

We have a corollary of Proposition \ref{prop:e}.

\begin{cor}\label{cor:Delta(e)}
	$\Delta(e)=p(k\delta^p+\nabla(a))$ for some integer $k$ with $p\nmid k$ and some $a\in\Lambda_p$, where $\delta\in K_p$ is defined in Section \ref{sec:BPU(p)}.
\end{cor}
\begin{proof}
	By Proposition \ref{prop:e}, $\Delta(e)\in pK_p$. Since 
	\[K_p=\oplus_{i\geq 0}\Zz\{\delta^i\}+K_p\cap\nabla(\Lambda_p)\] 
	by Corollary \ref{cor:delta}, $\Delta(e)=p(k\delta^p+\nabla(a))$ for some integer $k$ and some $a\in \Lambda_p$. We need to show that $p\nmid k$. Note that $\Theta(\delta)\equiv -1$ mod $p$ by \eqref{eq:Theta(delta)}, and that $\nabla(a)\in K_p$ since $\Delta(e),\delta\in K_p$. Hence, we have $\Theta\Delta(e)\equiv -pk$ mod $p^2$ by Corollary \ref{cor:I_p}. 
	
	On the other hand, the coefficient $A$ of $\sigma_{p^2-p}^p$ in $e$ satisfies $v_p(A)=1$ by the proof of Proposition \ref{prop:e}. It follows that
	\[\Delta(e)\equiv pk'\sigma_{p-1}^{p^2}+\text{other monomials} \bmod p^2,\text{ for some }k'\text{ with }p\nmid k',\]
	since by \eqref{eq:diagonal map} $\Delta(\sigma_{p^2-p})\equiv\sigma_{p-1}^p$ mod $p$, and if $\mm=\prod_k\sigma_{j_k}\in\Lambda_{p^2}^{p^3-p^2}$ is a monomial other than a multiple of $\sigma_{p^2-p}^p$, so that $j_k\neq p^2-p$ for at least two indices $k$, then any nonzero multiple of $\sigma_{p-1}^{p^2}$ does not appear in $\Delta(\mm)$ mod $p^2$.
	Thus we have $\Theta\Delta(e)\equiv-pk'$ mod $p^2$ by Lemma \ref{lem:Theta mod p}, using the fact that a monomial in $\Lambda_p^{p^3-p^2}$ other than a multiple of $\sigma_{p-1}^{p^2}$ contains at least two factors $\sigma_i\neq\sigma_{p-1}$. So $p\nmid k$. 
\end{proof}

\begin{prop}\label{prop:e not in}
	$e\notin\nabla(\Lambda_{p^2})$.
\end{prop}
\begin{proof}
	By Corollary \ref{cor:Delta(e)}, we can write $\Delta(e)=p(k\delta^p+\nabla(a))$ for some integer $k$ with $p\nmid k$ and some $a\in\Lambda_p$. 
Suppose that $e=\nabla(f)$ for some $f\in\Lambda_{p^2}$ and write $h=\Delta(f)$. Since $f$ has degree $p^3-p^2+1$, \eqref{eq:diagonal map} shows that $h\in p\Lambda_p$, say, $h=ph'$ for some $h'\in \Lambda_p$. Since $\nabla$ commutes with $\Delta$, it follows that 
\[\nabla(ph')=\nabla(h)=\nabla\Delta(f)=\Delta\nabla(f)=\Delta(e)=p(k\delta^p+\nabla(a)).\]
Thus $\nabla(h')=k\delta^p+\nabla(a)$, and then $k\delta^p\in\nabla(\Lambda_p)$. But this contradicts Lemma \ref{lem:delta} since $p\nmid k$.
\end{proof}

\section{On the map $\Delta:\overline K_{p^2}\to\overline K_p$}\label{sec:Delta for K_p^2}
In this section, we use invariant theory to study the image of the map (see \eqref{eq:the map Delta})
$\Delta:\overline K_{p^2}\to\overline K_p$, 
where $p$ is an arbitrary prime and $\overline K_n=\im(K_n\to\Lambda_n\otimes\Ff_p)$.

It is easy to see that $K_2=\Zz[\sigma_1^2-4\sigma_2]$. Moreover, for any integer $n\geq 3$ and any odd prime divisor $p$ of $n$, $\sigma_1^2,\sigma_1^3\in\overline K_n$ by \cite{Gu21} (see also \cite[(4.1)]{FZZZ24}), so that $\sigma_1^k\in \overline K_n$ for any $k\geq 2$. The following proposition is the main result of this section.
\begin{prop}\label{prop:L_p}
Let $\wh L_p$ be the subring of $\overline K_p$ generated by 
\[\sigma_1^{p^2}\ \text{ and }\ \{f^{p^2}:f\in \overline K_p\}.\]
Then $\wh L_p\subset \im\Delta$ for the map $\Delta:\overline K_{p^2}\to\overline K_p$.	
\end{prop}
Before starting the proof, let us recall an elementary result in invariant theory. Let $G$ be a finite group, and let $R$ be a ring with a left $G$-action. For an element $x\in R$, let $G_x=\{g\in G: gx=x\}$ be the isotropy group of $x$. For a subgroup $H\subset G_x$, define 
\[T^G_H(x)=\sum_{g\in G/H}gx,\]
where $G/H$ is the set of cosets of $H$ in $G$ and $g$ is an arbitrary representative in the coset $gH$.
\begin{lem}\label{lem:invariants}
	In the above notation, we have $T^G_H(x)\in R^G$ for any $x\in R$ and any subgroup $H\subset G_x$.
\end{lem}
\begin{proof}
For any $f\in G$, left multiplication permutes the left cosets of $H$. So
\[
fT_H^G(x)
=\sum_{gH\in G/H}(fg)x
=\sum_{eH\in G/H}ex
=T_H^G(x).
\]
Thus $T_H^G(x)\in R^G$.
\end{proof}
We prove Proposition \ref{prop:L_p} using the following three technical lemmas, whose proofs are based on suggestions from AI.

For an integer $n>0$, set 
\[B_n=\Zz[\bar v_1,\dots,\bar v_{n-1}]\subset\Zz[n],\ \ \bar v_i=v_i-v_n.\] 
In other words, $B_n$ is the kernel of the derivation $\nabla:\Zz[n]\to\Zz[n]$. Let $\MM_n$ be the set of products of the form
\[\prod_{1\leq i<j\leq n}(v_i-v_j)^{a_{i,j}},\ \ a_{i,j}\in \Zz_{\geq 0}.\] 
It is clear that elements in $\MM_n$ additively generate $B_n$.
For the symmetric group $G=S_n$, Lemma \ref{lem:invariants} shows that $T^G_{G_w}(w)\in B_n^G=\Lambda_n\cap B_n=K_n$ for any $w\in\MM_n$. Here the action of $G$ on $B_n$ is inherited from its canonical action on $\Zz[n]$.
\begin{lem}\label{lem:span of K_p}
	For any prime $p$,  the set
	\[\{T^G_{G_w}(w)\bmod p:w\in\MM_p\},\ \ G=S_p,\]
	spans $\overline K_p$.
\end{lem}
\begin{proof}
	We denote by $C\subset G$ the subgroup generated by the cycle $\tau=(1,2,\dots,p)$. Consider the $C$-equivariant surjection	
	\[\pi:S=\Zz[x_1,\dots,x_p]\to B_p,\ \ x_i\mapsto v_i-v_{i+1},\]
	where indices are cyclic (i.e., $v_{p+1}=v_1$). Its kernel is generated by $s=x_1+\cdots+x_p$, so $B_p=S/\langle s\rangle$.
	
	First, the map $S^C \to B_p^C$ is surjective. Indeed, if
	$u\in B_p^C$ and $\pi(f)=u$ for some $f\in S$, then
	$(\tau-1)f=sh$ for some $h\in S$.
	Letting $N=1+\tau+\cdots+\tau^{p-1}$ gives $sN(h)=0$,
	hence $N(h)=0$. We claim that $\ker N=\im(\tau-1)$ in $S$. Assuming this for the moment, $h=(\tau-1)b$ for some $b\in S$, and $f-sb\in S^C$, $\pi(f-sb)=u$, proving the surjectivity of $S^C \to B_p^C$ .
	
	Using group homology, the claim is equivalent to saying that $H_2(C,S)=0$. Note that the standard resolution of the trivial
	$\Zz[C]$-module $\Zz$ is
	\[
	\cdots
	\xr{\tau-1}\Zz[C]\xr{N}\Zz[C]
	\xr{\tau-1}\Zz[C]
	\xr{\varepsilon}\Zz
	\to0.
	\]
	After tensoring over $\Zz[C]$ with $S$, we obtain
	$H_2(C,S)=\ker N/\im(\tau-1)$. Let $\OO$ be the set of $C$-orbits of monic monomials in $S$. Then there is a decomposition of $S$ into a direct sum of $\Zz[C]$-modules
	\[S\cong \bigoplus_{[\mm]\in\OO}\Zz[C/C_\mm],\]
	where $\mm$ is an arbitrary representative in its orbit $[\mm]$.
	Each summand on the right-hand side is either $\Zz$ or $\Zz[C]$, and both cases contribute zero to $H_2(C,S)$, hence $H_2(C,S)=0$.
	
	Given $u\in B_p^{C}$, choose a polynomial $f\in S^C$ such that $\pi(f)=u$, and write $f$ in the monomial basis. Since $f$ is $C$-invariant,
	the coefficients of monomials in the same $C$-orbit are equal.
	Therefore,
	\begin{equation}\label{eq:monomial-orbit-decomposition}
		f=\sum_{I\in\OO}a_I\left(\sum_{\mm\in I}\mm\right),\quad a_I\in\Zz.
	\end{equation}
	Now choose a representative
	$\mm=\prod_{i=1}^p x_i^{b_i}$ in an orbit $I$, and put
	\[
	w=\pi(\mm)=\prod_{i=1}^p(v_i-v_{i+1})^{b_i},\ \ v_{p+1}=v_1.
	\]
	Equivariance of $\pi$ implies
	$C_\mm\subset C_w$. Consequently,
	\begin{equation}\label{eq:image-of-orbit-sum}
		\pi\left(\sum_{\mm\in I}\mm\right)=T^C_{C_\mm}(w)=[C_w:C_\mm]\,T^C_{C_w}(w).
	\end{equation}
	Applying $\pi$ to \eqref{eq:monomial-orbit-decomposition}
	and using \eqref{eq:image-of-orbit-sum}, we see that $B_p^C$ lies in the integral span of $\{T^C_{C_w}(w):w\in\MM_p\}$. 
	
	Finally, for any $w\in\MM_p$, we have
	\[
	T_C^G(T^C_{C_w}(w))=T_C^G\left(\sum_{g\in C/C_w}gw\right)
	=\sum_{g\in G/C_w}gw
	=\frac{|G_w|}{|C_w|}\,T_{G_w}^G(w).
	\]
	Thus $T_C^G(B_p^C)$ lies in the integral span of $\{T^G_{G_w}(w):w\in\MM_p\}$.
	However, for every $u\in K_p=B_p^G\subset B_p^C$,
	\[
	T_C^G(u)=[G:C]u=(p-1)!\,u.
	\]
	Since $(p-1)!$ is invertible modulo $p$, $\{T^G_{G_w}(w)\bmod p:w\in\MM_p\}$ spans $\overline K_p$.
\end{proof}

Given $w=\prod_{1\leq i<j\leq p}(v_i-v_j)^{a_{i,j}}\in\MM_p$, define
\[\phi_w=\prod_{k=0}^{p-1}\prod_{l=0}^{p-1}\prod_{1\leq i<j\leq p}(v_{kp+i}-v_{lp+j})^{a_{i,j}}\in\Zz[p^2].\]
\begin{lem}\label{lem:phi_w}
	Let $p$ be an odd prime and let $\Delta:\Zz[p^2]\to\Zz[p]$ be the map as in \eqref{eq:the map Delta}. Set $G=S_p$, $G'=S_{p^2}$. Then for any $w\in\MM_p$, we have
	\[\Delta(T^{G'}_{G'_{\phi_w}}(\phi_w))\equiv T^G_{G_w}(w)^{p^2}\bmod p.\]	
\end{lem}
\begin{proof}
	Extend the notation $\phi_w$ to signed products by setting $\phi_{-w}=-\phi_w$. For $g\in G$, let $\tilde g\in G'$ be the permutation  \[\tilde g(kp+i)=kp+g(i),\ \ 0\leq k\leq p-1,\ 1\leq i\leq p.\]
	Since $p$ is odd, $\phi_{-\omega}=(-1)^{p^2}\phi_\omega=\phi_{g\omega}$ for any difference $\omega=v_i-v_j$ and the transposition $g=(i,j)\in G$. It follows that 
	\begin{equation}\label{eq:phi_w}
		\Delta(\phi_w)=w^{p^2}\ \text{ and }\ \tilde g\phi_w=\phi_{g w}.
	\end{equation}
	For $w\in\MM_p$, let $\II_w$ be the set of distinct elements in 
	$\{g\phi_w:g\in G'\}$, so that
	\begin{equation}\label{eq:orbit sum}
		T^{G'}_{G'_{\phi_w}}(\phi_w)=\sum_{\phi\in\II_w} \phi.
	\end{equation}
	
	Partition $[p^2]$ into 
	\[[p^2]=\bigcup_{i=1}^p U_i,\ \ U_i=\{kp+i:0\leq k\leq p-1\}.\]
	Consider the subgroup $H=C_p^{\times p}=C_p\times\cdots\times C_p\subset G'$, where $C_p$ is the cyclic group of order $p$ and the $i$th factor cyclically permutes elements of $U_i$. Since $\Delta(v_r)=v_i$ for any $r\in U_i$, it follows that 
	$\Delta(h\phi)=\Delta(\phi)$ for any $h\in H$ and any $\phi\in\Zz[p^2]$, hence the sum over each nontrivial $H$-orbit maps to zero under $\Delta$ modulo $p$. From this and \eqref{eq:orbit sum}, we see that
	\begin{equation}\label{eq:orbit mod p}
		\Delta(T^{G'}_{G'_{\phi_w}}(\phi_w))\equiv\sum_{\phi\in \II_w^H}\Delta(\phi)\bmod p,
	\end{equation}
	where $\II_w^H=\{\phi\in \II_w:\phi\text{ is fixed by }H\}$. Clearly, $\phi_w\in \II_w^H$ for any $w\in\MM_p$.
	
	For $w\in\MM_p$, let $\JJ_w$ be the set of distinct elements in $\{gw:g\in G\}$. We claim that
	\begin{equation}\label{eq:I_w=J_w}
		\{\phi\in\mathcal \II_w^H:\Delta(\phi)\neq0\}=\{\phi_u:u\in \JJ_w\}.
	\end{equation}
	Assuming this is true and using \eqref{eq:phi_w}, \eqref{eq:orbit mod p}, we then have
	\[\Delta(T^{G'}_{G'_{\phi_w}}(\phi_w))\equiv\sum_{u\in\JJ_w}\Delta(\phi_u)=\sum_{u\in\JJ_w}u^{p^2}\equiv T^G_{G_w}(w)^{p^2}\bmod p.\]
	This proves the lemma.
	
	It remains to prove \eqref{eq:I_w=J_w}. Note that $\phi_{gw}$ is fixed by $H$ for any $w\in\MM_p$ and any $g\in G$, so the inclusion from right to left follows immediately from the definitions and \eqref{eq:phi_w}. Now we prove the converse.
	For a signed product $\phi$ of $(v_i-v_j)$'s in $\Zz[p^2]$, let $\phi_{i,j}$ denote the exponent of the factor $v_i-v_j$, ignoring its sign, and set $\phi_{i,i}=0$. Then the  matrix $M_\phi=(\phi_{i,j})$ is a $p^2\times p^2$ symmetric matrix with diagonal entries zero. Write $\rr_i(\phi)\in\Zz^{p^2}$ for the $i$th row vector of $M_\phi$. 
	
	Suppose $\phi\in\II_w^H$ and $\Delta(\phi)\neq0$. Then $\phi$ does not have a factor $v_i-v_j$ with $i,j\in U_k$ for some $k$, since such a factor becomes zero under $\Delta$. Moreover, $H$-invariance implies that $\phi_{i,j}$ is constant as $(i,j)$ ranges over $U_k\times U_l$, for $k\neq l$. Hence for every $U_k$ and any $i,j\in U_k$, $\rr_i(\phi)=\rr_j(\phi)$.
	Assume that the set of distinct row vectors $\rr_i(\phi)$, $1\leq i\leq p^2$, is $\{\rr_1,\dots,\rr_s\}$, and define the row classes
	\[N_i=\{j\in[p^2]:\rr_j(\phi)=\rr_i\},\ \ 1\leq i\leq s.\]
	Then $[p^2]=\bigcup_{i=1}^s N_i$. 
	The above argument shows that each $N_i$ is a union of some $U_k$'s.
	
	Now choose $g\in G'$ with $\phi=g\phi_w$, so that $(\phi_w)_{i,j}=\phi_{g(i),g(j)}$ for any $i,j\in[p^2]$. Since $\rr_i(\phi_w)=\rr_j(\phi_w)$ for every $U_k$ and any $i,j\in U_k$, it follows that $\rr_{g(i)}(\phi)=\rr_{g(j)}(\phi)$ for the same $U_k$ and $i,j\in U_k$. So each $N_i$ is also a union of some $g(U_k)$'s. 
	Thus we may choose a permutation $\sigma\in G$ such that $g(U_i)$ and $U_{\sigma(i)}$ belong to the same row class for every $i\in[p]$.
	For any $j=g(kp+i)\in g(U_i)$, we have $\tilde\sigma g^{-1}(j)=kp+\sigma(i)\in U_{\sigma(i)}$.
	So $\tilde\sigma g^{-1}$ preserves every row class.
	
	Consider the subgroup $E=S_{n_1}\times\cdots\times S_{n_s}\subset G'$, $n_i=|N_i|$, with each factor $S_{n_i}$ permuting the elements of $N_i$.
	Given two elements $i,j$ in the same row class, we have $\phi_{i,j}=\phi_{j,j}=0$ and
	$\phi_{i,k}=\phi_{j,k}$ for any $k\neq i,j$. Up to a sign, the factors of $\phi$ involving $v_i$ or $v_j$ can therefore be written as
	\[\prod_{k\neq i,j} \bigl((v_i-v_k)(v_j-v_k)\bigr)^{\phi_{i,k}},\]
	which is unchanged by interchanging $i$ and $j$. Since transpositions generate symmetric groups, we have $g\phi=\phi$ for every $g\in E$.
	So the above permutation $\tilde\sigma g^{-1}\in E$ fixes $\phi$. Therefore,
	\[\phi=(\tilde\sigma g^{-1})\phi =\tilde\sigma\phi_w =\phi_{\sigma w}.\]
	This completes the proof of \eqref{eq:I_w=J_w}.
\end{proof}

\begin{lem}\label{lem:the element f}
For any prime $p$ and any integer $n>0$ with $p^2|n$, there is an element $f\in K_n^{p^2}$ such that the coefficient of $\sigma_p^p$ in $f$ is $1$ modulo $p$.	
\end{lem}
\begin{proof}
For each $p$-element subset $A\subset[n]$, put
\[r_A=\sum_{i\in A}v_i,\quad f_A=\frac{1}{p}\sum_{i\in A}(r_A-pv_i)^p.\]
We see that $f_A\in (\ker\nabla)_{(p)}$. In fact, expansion of $f_A$ gives
\begin{equation}\label{eq:f_A}
f_A=r_A^p+\sum_{k=1}^{p}\binom pk(-1)^kp^{k-1} r_A^{p-k}\sum_{i\in A}v_i^k,
\end{equation}
so $f_A\in \ker\nabla$.
In particular,
\begin{equation}\label{eq:f_A mod p}
	f_A\equiv r_A^p\equiv\sum_{i\in A}v_i^p\bmod p.
\end{equation}

Let $\Aa$ be the set of unordered collections of $p$ pairwise disjoint $p$-element subsets of $[n]$. Since $p^2|n$, $\Aa\neq\emptyset$. Define
\begin{equation}\label{eq:f}
f=\frac{1}{p^{p-1}}\sum_{\{A_1,\ldots,A_p\}\in\Aa} f_{A_1}\cdots f_{A_p}.
\end{equation}
$f$ is clearly homogeneous of degree $p^2$ and symmetric, so that $f\in (K_n^{p^2})_{(p)}$. We will show that $f$ has integral coefficients and that the coefficient of $\sigma_p^p$ in $f$ is congruent to $1$ modulo $p$.

Suppose $\mm=t\prod_{j=1}^kv_{i_j}^{a_j}$, $0\neq t$, $0<a_j\in\Zz$, is a monomial in the expansion of a summand $\prod_{i=1}^p f_{A_i}$ on the right side of \eqref{eq:f}. Define $s(\mm)=\{i_1,\ldots,i_k\}$, and let $s_i=|s(\mm)\cap A_i|$. We have $1\leq s_i\leq p$. Let $l$ be the cardinality $\#\{i\in[p]:s_i\geq 2\}$. By \eqref{eq:f_A} and \eqref{eq:f_A mod p}, we obtain that $p^l|t$.

Fix $p$  pairwise disjoint subsets $W_1,\ldots,W_p$ of $[n]$ with $1\leq |W_i|\leq p$ for each $i$, and let $s_i=|W_i|$, $s=\sum_{i=1}^p s_i$. The number of collections $\{A_1,\ldots,A_p\}\in\Aa$ such that $W_i\subset A_i$ is 
\begin{equation}\label{eq:number of choices}
	m=\frac{(n-s)!}{(n-p^2)!\prod_{i=1}^p(p-s_i)!}.
\end{equation}
To see this, write $A_i=W_i\sqcup Z_i$, where the $Z_i$ are pairwise disjoint subsets of the complement
\[Z=[n]\setminus\bigcup_{i=1}^p W_i,\quad |Z|=n-s,\ \ |Z_i|=p-s_i.\]
Thus we choose $p-s_1$ elements for $Z_1$, then $p-s_2$ from the remaining elements for $Z_2$, and so on. The number of choices is
\[
	\begin{split}
		m=&\binom{n-s}{p-s_1}\binom{n-s-(p-s_1)}{p-s_2}\cdots
		\binom{n-s-\sum_{i=1}^{p-1}(p-s_i)}{p-s_p}\\
		=&\frac{(n-s)!}
		{\left(n-s-\sum_{i=1}^{p}(p-s_i)\right)!\prod_{i=1}^{p}(p-s_i)!}\\
		=&\frac{(n-s)!}
		{(n-p^2)!\prod_{i=1}^{p}(p-s_i)!}.
	\end{split}	
\]
	
Since every $(p-s_i)!$ is prime to $p$,
\[v_p(m)=v_p\left(\frac{(n-s)!}
{(n-p^2)!}\right)=v_p\left(\prod_{j=1}^{p^2-s}(n-p^2+j)\right)=\lfloor \frac{p^2-s}{p}\rfloor=p-\lceil\frac{s}{p}\rceil,\]
where the third equality follows from the fact that $p^2|n$. Moreover,
\[s=\sum_{i=1}^p s_i\leq (p-l)+pl=p+(p-1)l,\ \ l=\#\{i\in[p]:s_i\geq 2\},\]
so $\lceil s/p\rceil\leq l+1$ and $v_p(m)+l\geq p-1$. This, combined with the previous paragraph, implies that the coefficient of a monomial $\mm$ in the polynomial \[\sum_{\{A_1,\ldots,A_p\}\in\Aa} f_{A_1}\cdots f_{A_p}\] is a multiple of $p^{p-1}$. So $f$ is integral.

Finally, we compute the coefficient modulo $p$ of $\sigma_p^p$ in $f$. Using the evaluation
\[
v_1=\cdots=v_p=1,\ \ v_{p+1}=\cdots=v_n=0,
\]
a nonzero summand on the right side of \eqref{eq:f} must have exactly one of the $p$
variables $v_1,\dots,v_p$ in each $A_i$. For such a product $f_{A_1}\cdots f_{A_p}$, the evaluation of each factor is
\[
f_{A_i}=a,\ \ a=\frac{(1-p)^p+p-1}{p}\equiv 1\bmod p.
\]
By \eqref{eq:number of choices}, there are
\[
\frac{(n-p)!}{(n-p^2)!((p-1)!)^p}
\]
such collections in $\Aa$. Consequently,
\begin{equation}\label{eq:f mod p} 
f(\underbrace{1,\ldots,1}_{p},
\underbrace{0,\ldots,0}_{n-p})
=
\frac{(n-p)!}
{p^{p-1}(n-p^2)!((p-1)!)^p}\,a^p\equiv1\bmod p.
\end{equation}
To verify this congruence modulo $p$, write $n-p^2=bp^2$. Then
\[
\frac{(n-p)!}{p^{p-1}(n-p^2)!}
=
\left(
\prod_{j=1,p\nmid j}^{p^2-p}
(bp^2+j)
\right)
\left(
\prod_{k=1}^{p-1}(bp+k)
\right).
\]
Reducing modulo $p$, we obtain
\[
\frac{(n-p)!}{p^{p-1}(n-p^2)!}
\equiv
((p-1)!)^{p-1}(p-1)!
=
((p-1)!)^p
\bmod p.
\]
Since $(p-1)!$ is invertible modulo $p$, it follows that
\[
\frac{(n-p)!}
{p^{p-1}(n-p^2)!((p-1)!)^p}
\equiv 1\bmod p.
\]
Combined with $a\equiv1$ mod $p$, this proves \eqref{eq:f mod p}.
On the other hand, the evaluations of $\sigma_i\in\Lambda_n$ at $(1^p,0^{n-p})$ are
\[
\sigma_i=\binom{p}{i}\equiv0\bmod p
\ \ (1\leq i<p),
\ \ 
\sigma_p=1,
\ \ 
\sigma_i=0
\ \ (i>p).
\]
Since $\deg(f)=p^2$, the only monomial in $\sigma_i$'s that survives modulo $p$ is $\sigma_p^p$. Thus \eqref{eq:f mod p} shows that the coefficient of
$\sigma_p^p$ in $f$ is $1$ modulo $p$.
\end{proof}
 We are now ready to prove our main result of this section. 
\begin{proof}[Proof of Proposition \ref{prop:L_p}]
For $p>2$, Lemmas \ref{lem:span of K_p} and \ref{lem:phi_w} together show that $\im\Delta$ contains the subring of $\overline K_p$ generated by $\{f^{p^2}:f\in \overline K_p\}$. Moreover, from Lemma \ref{lem:K_p^r} we see that for any prime $p$, $\Delta(f)\equiv\sigma_1^{p^2}$ mod $p$ for the element $f\in K_{p^2}^{p^2}$ given in Lemma \ref{lem:the element f}, and the proposition follows, using the fact that $\overline K_2=\Ff_2[\sigma_1^2]$ for $p=2$.
\end{proof}

\section{Elementary abelian $p$-subgroups of $PU(n)$}\label{sec:p-subgroups}
For any prime $p$, the nontoral elementary abelian $p$-subgroups of $PU(n)$ with $p|n$ are studied in \cite{AGMV08} and \cite{Griess91}. Recall that a subgroup of a compact Lie group $G$ is called \emph{toral} if it is contained in a torus in $G$ and \emph{nontoral} otherwise. For the rest of this section, we assume that $p$ is an odd prime.
\subsection{The group $PU(p)$}\label{subsec:PUp}
First we consider the case $n=p$.
Let $\omega=e^{2\pi i/p}$, and let $\tilde\sigma,\tilde\tau\in U(p)$ be
\[\tilde\sigma=\begin{pmatrix}
	0& 1\\
	I_{p-1} &0
\end{pmatrix},\quad
\tilde\tau=(\omega,\omega^2,\dots,\omega^{p-1},1)\in T^p.\]
Let $\tilde\Gamma$ be the subgroup of $U(p)$ generated by $\tilde\sigma$ and $\tilde\tau$, and let $\sigma$, $\tau$, $\Gamma$ denote the corresponding images in the quotient group $PU(p)$. We have $\tilde\tau\tilde\sigma=\omega\tilde\sigma\tilde\tau$, so $\sigma$ and $\tau$ commute in $PU(p)$, i.e. $\Gamma\cong (\Zz/p)^2$. 
Hence 
\[H^*(B\Gamma;\Ff_p)\cong \Ff_p[\xi,\eta]\otimes\Lambda_{\Ff_p}[a,b],\] 
where $a,b\in H^1(B\Gamma;\Ff_p)\cong \Hom(\Gamma,\Zz/p)$ are defined by $a(\sigma)=b(\tau)=1$, $a(\tau)=b(\sigma)=0$, and $\xi=\beta(a)$, $\eta=\beta(b)$. 
By a result of Andersen et al. (see Theorem \ref{thm:Weyl group} below), we have $W_{PU(p)}(\Gamma)=SL_2(\Ff_p)$.

Let $y=ab\in H^*(B\Gamma;\Ff_p)$.
Since $SL_2(\Ff_p)$ acts naturally on $\Ff_p\{a,b\}$, one checks that $y\in H^*(B\Gamma;\Ff_p)^{SL_2(\Ff_p)}$. Since the Bockstein homomorphism and Steenrod powers are $SL_2(\Ff_p)$-equivariant, the following elements
\[
s=\beta(y)=\xi b-\eta a,\ z=P^1(s)=\xi^pb-\eta^pa,\ f=\beta(z)=\xi^p\eta-\eta^p\xi
\]
lie in $H^*(B\Gamma;\Ff_p)^{SL_2(\Ff_p)}$. Moreover, let 
\[h=\xi^{p^2-p}+\eta^{p-1}(\xi^{p-1}-\eta^{p-1})^{p-1}.\]
Then by the result of Dickson  \cite{Dick11} for the invariants of the action of $SL_n(\Ff_p)$ on $\Ff_p[n]$ (here we use it for the special case $n=2$), we have
\begin{equation}\label{eq:Dickson}
\Ff_p[\xi,\eta]^{SL_2(\Ff_p)}=\Ff_p[f,h]\subset H^*(B\Gamma;\Ff_p)^{SL_2(\Ff_p)}.
\end{equation}

The following proposition is the special case $n=2$ of Mui's theorem \cite{Mui75} (see also \cite[Theorem 1.1]{KM07}) for the ring of $SL_n(\Ff_p)$-invariants of the tensor product algebra $\Ff_p[n]\otimes\Lambda_{\Ff_p}[n]$. 
\begin{prop}\label{prop:Mui}
	$H^*(B\Gamma;\Ff_p)^{SL_2(\Ff_p)}$ is the subring of $H^*(B\Gamma;\Ff_p)$ generated by $f,h,s,y,z$, which is isomorphic to the following quotient ring
	\[R=\frac{\Ff_p[f,h]\otimes\Lambda_{\Ff_p}[s,y,z]}{\langle ys,\,yz,\,fy+sz\rangle}.\]
\end{prop}
\begin{cor}[{\cite[Proposition 5.10]{Vis07}}]\label{cor:cohomology of Gamma}
	\[H^*(B\Gamma)^{SL_2(\Ff_p)}=\Zz[f,h,s]/\langle pf,ph,ps,s^2\rangle.\]
\end{cor}
Let $i:\Gamma\to PU(p)$ be the inclusion and consider the restriction map of mod $p$
cohomology
\[(Bi)^*:H^*(BPU(p);\Ff_p)\to H^*(B\Gamma;\Ff_p).\]
By \cite[(3.1), (3.2)]{Fan26b}, we have
\begin{gather}\label{eq:f,h}
	\begin{gathered}
	(Bi)^*(\chi)=s,\ \ (Bi)^*(\gamma)=f,\ \ (Bi)^*(\delta)=-h,\\ 
	(Bi)^*(\sigma_1)=y,\ \ (Bi)^*(\zeta)=z,
	\end{gathered}
\end{gather}
where $\chi$, $\gamma$, $\delta$, $\sigma_1$, $\zeta$ are defined in Section \ref{sec:BPU(p)}.
\subsection{The group $PU(n)$}\label{subsec:p-subgroups}
 For any integer $m>0$, a canonical nontoral elementary abelian $p$-subgroup of $PU(p^m)$ is constructed in \cite{Oliver94} (see also \cite[Section 3]{Gu26}). To describe this elementary abelian $p$-subgroup, we use the $m$-fold tensor product homomorphism:
\[t:{U(p)\times \cdots \times U(p)}\to U(p^m),\ \ t(g_1,\dots,g_m)=g_1\otimes\cdots\otimes g_m.\]
Let $\tilde\sigma,\tilde\tau\in U(p)$ and $\w\Gamma\subset U(p)$ be as before, and let $\w N_m$ be the subgroup of $U(p^m)$ generated by elements $g_1\otimes\cdots\otimes g_m$, $g_i\in\w\Gamma$. The center of $\w N_m$ is the subgroup $\Zz/p$ of $U(p^m)$ generated by the scalar matrix $\omega I_{p^m}$. So  $\tilde N_m$ fits into a short exact sequence of groups:
\[1\to\Zz/p\to\w N_m\to N_m\to 1,\ \ N_m\subset PU(p^m).\]

Define
$\tilde\sigma_i=1\otimes\cdots\otimes1\otimes\tilde\sigma\otimes1\otimes\cdots\otimes1$,
where $\tilde\sigma$ occupies the $i$th entry. $\tilde\tau_i$ is similarly defined with $\tilde\sigma$ replaced by $\tilde\tau$. Let $\sigma_i$, $\tau_i$ be the images of $\tilde\sigma_i$, $\tilde\tau_i$ in $PU(p^m)$, and let $\Gamma_i$ be the subgroup of $PU(p^m)$ generated by $\sigma_i,\tau_i$. 
\begin{prop}[{\cite[pp. 56-57]{Oliver94}}]\label{prop:N_m}
	$N_m$ is a nontoral elementary abelian $p$-subgroup of $PU(p^m)$ isomorphic to
	\[\Gamma_1\times\cdots\times\Gamma_m\cong (\Zz/p)^{2m}.\]
\end{prop}
Regarding $N_m$ as a vector space over $\Ff_p$, there is a canonical symplectic form
\[\Omega_m:N_m\times N_m\cong \Ff_p^{2m}\times\Ff_p^{2m}\to \Ff_p\]
defined as follows. If $g,g'\in N_m\subset PU(p^m)$, lift them to matrices $\tilde g,\tilde g'$ in $U(p^m)$. Then the commutator $[\tilde g,\tilde g']=\omega^r I_{p^m}$ for a unique $r\in\Ff_p$. Write $\Omega_m(g,g')=r$.  This alternating bilinear form is independent of the liftings and it is a  symplectic form with a standard symplectic basis $\{\sigma_1,\dots,\sigma_m,\tau_1,\dots,\tau_m\}$ such that
\[\Omega_m(\tau_i,\sigma_j)=\delta_{ij},\ \ \Omega_m(\sigma_i,\sigma_j)=\Omega_m(\tau_i,\tau_j)=0.\] 

More generally, suppose that $n>0$ is an integer with $p$-primary factor $p^m$, $m\geq 1$. For each $1\leq i\leq m$, let $N_i\subset PU(p^i)$ be defined as above. 
Let $n_i=n/p^i$, and let $A_i$ be a maximal toral elementary abelian $p$-subgroup of $PU(n_i)$. The tensor product map $T:U(p^i)\times U(n_i)\to U(n)$, $T(g,g')=g\otimes g'$, induces a homomorphism:
\[T:PU(p^i)\times PU(n_i)\to PU(n).\]
It turns out that $E_i:=T(N_i\times A_i)\cong N_i\times A_i$ is a nontoral elementary abelian $p$-subgroup of $PU(n)$.

\begin{thm}[{\cite[Theorem 8.5]{AGMV08}}]\label{thm:Weyl group}
Suppose that $n>0$ is an integer with $p$-primary factor $p^m$, $m\geq 1$. Then there are exactly $m$ conjugacy classes $C_1,\dots,C_m$ of maximal nontoral elementary abelian $p$-subgroups of $PU(n)$. For each $1\leq i\leq m$, the group $E_i$ defined above gives a representative of $C_i$.

The Weyl groups are
\[W_{PU(n)}(E_i)=\begin{pmatrix}
		Sp_{2i}(\Ff_p)&0\\
		*&W_{PU(n_i)}(A_i)
	 \end{pmatrix}.\]
Here $Sp_{2i}(\Ff_p)$ is the symplectic group associated to the symplectic form $\Omega_i$ on $N_i$, and the symbol $*$ denotes an $(n_i-1)\times 2i$ matrix with arbitrary entries.
\end{thm}
The mod $p$ cohomology of $BE_i\simeq BN_i\times BA_i$ is given by
\[\begin{split}
	H^*(BE_i;\Ff_p)&\cong H^*(BN_i;\Ff_p)\otimes H^*(BA_i;\Ff_p)\\
	&\cong\big(\Ff_p[\xi_1,\eta_1,\dots,\xi_i,\eta_i]\otimes\Lambda_{\Ff_p}[a_1,b_1,\dots,a_i,b_i]\big)\\
	&\otimes\big(\Ff_p[v_1,\dots,v_{n_i-1}]\otimes\Lambda_{\Ff_p}[t_1,\dots,t_{n_i-1}]\big),
\end{split}\]
where $a_j,b_j\in H^1(BN_i;\Ff_p)\cong \Hom(N_i,\Zz/p)$ are defined by \[a_j(\sigma_k)=b_j(\tau_k)=\delta_{jk},\ \ a_j(\tau_k)=b_j(\sigma_k)=0,\] $\xi_j=\beta(a_j)$, $\eta_j=\beta(b_j)$, and $\{t_j\}_{j=1}^{n_i-1}$ is a basis of $H^1(BA_i;\Ff_p)$, $v_j=\beta(t_j)$. 

Let $\Phi_i$ be the restriction map 
\[\Phi_i:H^*(BPU(n);\Ff_p)\to H^*(BE_i;\Ff_p),\]
and let $W_i=W_{PU(n)}(E_i)$.
The image of $\Phi_i$ is contained in $H^*(BE_i;\Ff_p)^{W_i}$, and the action of $W_i$ on $H^*(BE_i;\Ff_p)$ is induced by the linear transformations on the two generator spaces 
\begin{gather*}
	\Ff_p\{\xi_1,\eta_1,\dots,\xi_i,\eta_i,v_1,\dots,v_{n_i-1}\},\\ \Ff_p\{a_1,b_1,\dots,a_i,b_i,t_1,\dots,t_{n_i-1}\}
\end{gather*}	
given by the matrix in Theorem \ref{thm:Weyl group}.

\begin{lem}\label{lem:image of chi}
	$\Phi_i(\chi)=\sum_{j=1}^is_j$, where $s_j=\xi_jb_j-\eta_ja_j$.
\end{lem}
\begin{proof}
	Let $F_i=SL_2(\Ff_p)^{\times i}\subset Sp_{2i}(\Ff_p)$ be the $i$-fold direct product of $SL_2(\Ff_p)$.
	$F_i$ acts on the generator spaces $\prod_{j=1}^i\Ff_p\{\xi_j,\eta_j\}$ and $\prod_{j=1}^i\Ff_p\{a_j,b_j\}$ in the natural way such that the $k$th factor $SL_2(\Ff_p)$ acts on $\Ff_p\{\xi_k,\eta_k\}$, $\Ff_p\{a_k,b_k\}$.
	
	Let $S_i\subset Sp_{2i}(\Ff_p)$ be the symmetric group that permutes the $i$ factors $\Ff_p^2$ of $\Ff_p^{2i}\cong (\Ff_p^2)^{\times i}$. Namely, $g(\sigma_j)=\sigma_{g(j)}$, $g(\tau_j)=\tau_{g(j)}$ for any $g\in S_i$. Then $F_i\rtimes S_i\subset Sp_{2i}(\Ff_p)\subset W_i$, so we have
	\begin{equation}\label{eq:invariant}
		H^*(BE_i;\Ff_p)^{W_i}\subset H^*(BE_i;\Ff_p)^{F_i\rtimes S_i}\cong (R^{\otimes i})^{S_i}\otimes H^*(BA_i;\Ff_p),
	\end{equation}
	where $R$ is the ring defined in Proposition \ref{prop:Mui}. 
	
	Let $\mathbf{s}=\sum_{j=1}^is_j$, $\yy=\sum_{j=1}^ia_jb_j$. Recall that $\deg(\chi)=3$, and that $A_i$ is toral. Hence, by \eqref{eq:invariant} and Proposition \ref{prop:Mui}, $\Phi_i(\chi)$ has the form
	\[\Phi_i(\chi)=l\mathbf{s}+\yy\otimes\mathbf{t},\ l\in\Ff_p,\ \mathbf{t}\in H^1(BA_i;\Ff_p).\] 
	Since $\beta(\chi)=\beta(\mathbf{s})=0$ and $\beta(\yy)=\mathbf{s}$, we have $\mathbf{t}=0$. It remains  to show that $l=1$. This is true for $n=p$ by \eqref{eq:f,h}.
	For the general case, we use the fact that $\bar\Delta^*(\chi)=\chi$ for the  map $\bar\Delta:BPU(p)\to BPU(n)$ defined in \eqref{eq:bar Delta} and notice that $\bar\Delta^*(\mathbf{s})=s$ for the restriction $\bar\Delta:B\Gamma\to BE_i$. 
\end{proof}

When $n=p^m$, Theorem \ref{thm:Weyl group} shows that $W_m=Sp_{2m}(\Ff_p)$. In this case, let $P$ be the polynomial subring $\Ff_p[\xi_1,\eta_1,\dots,\xi_m,\eta_m]$ of $H^*(BE_m;\Ff_p)$.
Gu \cite[Proposition 5.4]{Gu26} proves that
\[\im\Phi_m\cap P=P^{W_m}.\]
In the next section we will discuss this ring of polynomial invariants of finite symplectic groups. 

\section{Polynomial invariants of finite symplectic groups}\label{sec:symplectic invariants}
Let $p$ be a prime and let $V$ be a vector space over $\Ff_p$ of  dimension $2m$ with basis $u_i,v_i$, $1\leq i\leq m$. We endow $V$ with a non-singular symplectic form $\langle-,-\rangle$ given by
\[\langle u_i,v_j\rangle=\delta_{ij},\ \ \langle u_i,u_j\rangle=\langle v_i,v_j\rangle=0.\]
Let $Sp_{2m}(\Ff_p)$ be the symplectic group associated to this symplectic form.
Denote by $\Ff_p[V]$ the $\Ff_p$-algebra of polynomial functions on $V$. We consider the ring $\Ff_p[V]^{Sp_{2m}(\Ff_p)}$ of $Sp_{2m}(\Ff_p)$-invariant polynomial functions.
 
Let $V^*$ be the dual vector space of $V$ with basis $\xi_i,\eta_i$, $1\leq i\leq m$, dual to $u_i,v_i$. It is easy to verify that
\begin{equation}\label{eq:f_i}
	f_i:=\sum_{j=1}^m\xi_j^{p^i}\eta_j-\eta_j^{p^i}\xi_j\in \Ff_p[V]^{Sp_{2m}(\Ff_p)},\ \ i\geq 1.
\end{equation}
Moreover, the Dickson invariants $C_{2m,i}\in \Ff_p[V]^{GL_{2m}(\Ff_p)}$, $0\leq i\leq 2m-1$, lie in $\Ff_p[V]^{Sp_{2m}(\Ff_p)}$, where $(-1)^iC_{2m,i}$ is defined as the coefficient of $X^{p^i}$ in the polynomial
\[\prod_{v\in V^*}(X-v)\in\Ff_p[V][X].\]
Since $V^*$ has $p^{2m}$ elements, $C_{2m,i}$ is a polynomial of degree $p^{2m}-p^{i}$. 

The following theorem was originally proved by Carlisle and Kropholler. 
\begin{thm}[{\cite[Theorem 8.3.11]{Ben93}}]\label{thm:symplectic invariants}
The ring $\Ff_p[V]^{Sp_{2m}(\Ff_p)}$ is generated by the elements
\[C_{2m,m},\ldots,C_{2m,2m-1},\ f_1,\ldots,f_{2m-1}.\]
\end{thm}
\begin{rem}
	The elements in Theorem \ref{thm:symplectic invariants} are not algebraically  independent for $m>1$. Actually, \cite[Theorem 8.3.11]{Ben93} gives a presentation of $\Ff_p[V]^{Sp_{2m}(\Ff_p)}$ with the above $3m-1$ generators  and $m-1$ relations.
	
	In particular, when $m=1$ we have $f_1=f$ and $C_{2,1}=h$, where $f,h$ are the elements in \eqref{eq:Dickson}.
\end{rem}
If we think of $\Ff_p[V]$ as the polynomial subring of $H^*(BE_m;\Ff_p)$ for the subgroup $E_m\subset PU(p^m)$ defined in Section \ref{subsec:p-subgroups}, then by the definition of $y_{p,i}$ and Lemma \ref{lem:image of chi} (see also \cite[Proposition 5.1]{Gu26}), we have
\begin{equation}\label{eq:Phi_m}
	\Phi_m(y_{p,i})=f_{i+1},\ i\geq 0,
\end{equation}
where $\Phi_m$ is the restriction map defined in Section \ref{subsec:p-subgroups}.

\section{The map $\Phi_1:H^*(BPU(p^2);\Ff_p)\to H^*(BE_1;\Ff_p)^{W_1}$}\label{sec:image of Phi_1}
In this section we study the restriction map
\[\Phi_1:H^*(BPU(p^2);\Ff_p)\to H^*(BE_1;\Ff_p)^{W_1}\]
defined in Section \ref{sec:p-subgroups} for any odd prime $p$. Recall that $E_1=\Gamma\times T\subset PU(p)\times PU(p)$, where $\Gamma$ is the unique (up to conjugation) maximal nontoral elementary abelian $p$-subgroup of the first $PU(p)$ and $T$ is a maximal toral elementary abelian $p$-subgroup of the second $PU(p)$. By Theorem \ref{thm:Weyl group}, $W_1=U\rtimes(SL_2(\Ff_p)\times W_{PU(p)}(T))$, where
\[U=M_{(p-1)\times 2}(\Ff_p)\cong \{\begin{pmatrix}
	I_2&0\\
	M&I_{p-1}
\end{pmatrix}:M\in M_{(p-1)\times 2}(\Ff_p)\}.\]
By \cite[Proposition 5.1]{ADL24}, $W_{PU(p)}(T)\cong W/C_W(T)$, where $W\cong S_p$ is the Weyl group of $PU(p)$. Furthermore, the discussion in \cite[Section 5.1]{ADL24} shows that $C_W(T)=1$ for $p>2$. Hence $W_1=U\rtimes(SL_2(\Ff_p)\times S_p)$.

Consider the subring $A\otimes B\subset H^*(BE_1;\Ff_p)$, where 
\begin{gather*}
	A=\Ff_p[\xi,\eta]\otimes\Lambda_{\Ff_p}[a,b]=H^*(B\Gamma;\Ff_p),\\ 
	B=\Ff_p[\bar v_{1},\dots,\bar v_{p-1}]\cong H^*(BT_{PU(p)};\Ff_p),
\end{gather*}
with
\[\pi^*(\bar v_i)=v_i-v_p\in H^*(BT_{U(p)};\Ff_p)=\Ff_p[v_1,\dots,v_p]\] 
for the quotient map $\pi:T_{U(p)}\to T_{PU(p)}$. 
Since the inclusion $E_1\subset PU(p^2)$ factors through $\Gamma\times T_{PU(p)}$, it follows that
\begin{equation}
	\label{eq:image of Phi_1}
	\im\Phi_1\subset (A\otimes B)^{W_1}.
\end{equation}

We first study the ring structure of $(A\otimes B)^{W_1}$. For an element 
\[M=\begin{pmatrix}
	a_1&\cdots&a_{p-1}\\
	b_1&\cdots&b_{p-1}
\end{pmatrix}^T\in U,\]
the action of $M$ on $A\otimes B$ is given by
\[M|_A=id,\ \ M(\bar v_i)=\bar v_i+a_i\xi+b_i\eta.\]
The action of $SL_2(\Ff_p)$ on $A$ is obvious and the action of $S_p$ on $B$ is inherited from its canonical action on 
\[H^*(BT_{U(p)};\Ff_p)=\Ff_p[v_1,\dots,v_p].\]
\begin{prop}\label{prop:the W_1 invariant ring}
Let $\wh L_p$ be the ring as in Proposition \ref{prop:L_p}, and let $R=A^{SL_2(\Ff_p)}$ be the ring as in Proposition \ref{prop:Mui}. Then there is a graded ring isomorphism 
\[(A\otimes B)^{W_1}\cong R\otimes \wh L_p.\]
\end{prop}

Before starting the proof, we introduce two lemmas about invariant subrings. 

In the polynomial ring $A[X]$, define
\[N(X)=\prod_{i,j\in\Ff_p}(X+i\xi+j\eta).\]
It can be shown that 
\begin{equation}\label{eq:N(X)}
	N(X)=X^{p^2}-hX^p+f^{p-1}X,
\end{equation}
where $f,h\in \Ff_p[\xi,\eta]$ are the polynomials in \eqref{eq:Dickson} (see \cite[Lemma 5.6]{Vis07}). This implies that $N(X)$ is $\Ff_p$-linear. Namely, if $a,b\in\Ff_p$, then 
\[N(aX+bY)=aN(X)+bN(Y).\]
\begin{lem}\label{lem:N(X)}
Suppose that $G=\Ff_p^2$ acts on $A[X]$ via
\[(i,j)a=a,\ \ (i,j)X=X+i\xi+j\eta,\ \ a\in A,\ (i,j)\in G.\]
Then $A[X]^G=A[N(X)]$.	
\end{lem}
\begin{proof}
Let $P=\Ff_p[\xi,\eta]$ be the polynomial subring of $A$. Since $G$ acts trivially on $A$ and $g(X)\in P[X]$ for any $g\in G$, it suffices to show that $P[X]^G=P[N(X)]$.

Let $\kk=\Ff_p(\xi,\eta)$ be the field of fractions of the polynomial ring $P$ and let $\kk_0=\kk(N(X))$.  
Consider the polynomial 
 \[q(Z)=N(Z)-N(X)\in\kk_0[Z].\] 
Since the roots of $q(Z)$ are exactly $X+i\xi+j\eta$, $i,j\in\Ff_p$, it follows that $\kk_1=\kk(X)$ is the splitting field of $q(Z)$. In particular $\kk_1$ is a Galois extension of $\kk_0$.

Since $\xi,\eta,N(X)$ are invariant under the action of $G$, $\kk_0$ is contained in the fixed field. So $G\subset\Gal(\kk_1/\kk_0)$. On the other hand, since $q(Z)$ is a polynomial of degree $p^2$ with a root $X$  and $\kk_1=\kk_0(X)$, we have \[|\Gal(\kk_1/\kk_0)|=[\kk_1:\kk_0]\leq p^2.\] 
So $\Gal(\kk_1/\kk_0)=G$, and hence $\kk_1^G=\kk_0$. Then
\[P[X]^G=\kk_1^G\cap P[X]=\kk_0\cap P[X].\]

Since $X$ is a root of the monic polynomial $q(Z)\in P[N(X)][Z]$, $P[X]$ is integral over $P[N(X)]$. Now $P[N(X)]$ as a polynomial algebra over $\Ff_p$ is integrally closed, and since $\kk_0$ is the field of fractions of $P[N(X)]$, we obtain $P[X]^G=P[N(X)]$.
\end{proof}
\begin{lem}\label{lem:L_p}
Let $L_p=\overline K_p[\sigma_1]\subset\Lambda_p\otimes\Ff_p$ be as in Theorem \ref{thm:mod p BPU(p)}. Then	$B^{S_p}=L_p$.
\end{lem}
\begin{proof}
From Remark \ref{rem:Weyl gorup invariant}, we know that the image of the restriction map 
\[H^*(BPU(p);\Ff_p)\to H^*(BT_{PU(p)};\Ff_p)=B\]
 is $B^{S_p}$. On the other hand, Theorem \ref{thm:mod p BPU(p)} implies that this image is just $L_p$.
\end{proof}
\begin{proof}[Proof of Proposition \ref{prop:the W_1 invariant ring}]
Since $W_1=U\rtimes(SL_2(\Ff_p)\times S_p)$, 
\[(A\otimes B)^{W_1}=[(A\otimes B)^U]^{SL_2(\Ff_p)\times S_p}.\]
Note that the action of $U\cong(\Ff_p^2)^{p-1}$ on $A\otimes B$ can be viewed as the direct product of the $p-1$ actions of $\Ff_p^2$, where the action of the $i$th factor is given by the action in Lemma \ref{lem:N(X)} with $X$ replaced by $\bar v_i$. Hence by Lemma \ref{lem:N(X)},
\[(A\otimes B)^U=A[N_1,\dots,N_{p-1}],\ \ N_i=N(\bar v_i).\]
Since $N_i$ is fixed by $SL_2(\Ff_p)$ by \eqref{eq:N(X)}, we then have
\[[(A\otimes B)^U]^{SL_2(\Ff_p)}=R[N_1,\dots,N_{p-1}],\ \ R=A^{SL_2(\Ff_p)}.\]

Recall that the action of $S_p$ on $B$ is given by $\tau(\bar v_i)=v_{\tau(i)}-v_{\tau(p)}$ for any $\tau\in S_p$ and its action on $A$ is trivial. Since $N(X)$ is $\Ff_p$-linear, it follows that for a permutation $\tau\in S_p$,
\[\begin{split}
	\tau(N_i)&=N(\tau(\bar v_i))=N(v_{\tau(i)}-v_{\tau(p)})\\
	&=N((v_{\tau(i)}-v_p)-(v_{\tau(p)}-v_p))=N_{\tau(i)}-N_{\tau(p)}.
\end{split}\]
Here we use the convention that $\bar v_p=N_p=0$. Hence there is an $S_p$-equivariant homomorphism
\[B\to \Ff_p[N_1,\dots,N_{p-1}],\ \ f(\bar v_1,\dots,\bar v_{p-1})\mapsto f(N_1,\dots,N_{p-1}).\]
It is easy to see that $N_1,\dots, N_{p-1}$ are algebraically independent, so this is a (ungraded) ring isomorphism. 

On the other hand, the composition of $B\to \Ff_p[N_1,\dots,N_{p-1}]$ with the projection $\Ff_p[N_1,\dots,N_{p-1}]\to B$ is just the $2$-fold Frobenius endomorphism
\[F^2:B\to B,\ \ f\mapsto f^{p^2}.\]
It follows that the homomorphism
\[F^2(B)\to \Ff_p[N_1,\dots,N_{p-1}],\ \ f(\bar v_1,\dots,\bar v_{p-1})^{p^2}\mapsto f(N_1,\dots,N_{p-1})\]
is a graded ring isomorphism. It restricts to an isomorphism
\[F^2(B^{S_p})\to \Ff_p[N_1,\dots,N_{p-1}]^{S_p}.\]
Furthermore, $B^{S_p}=L_p$ by Lemma \ref{lem:L_p}, so by the definition of $\wh L_p$ we have $\wh L_p=F^2(B^{S_p})$. Hence there is a graded ring isomorphism 
\[\wh L_p\cong \Ff_p[N_1,\dots,N_{p-1}]^{S_p}.\]
Combining this with the fact that
\[(A\otimes B)^{W_1}=R[N_1,\dots,N_{p-1}]^{S_p}=R\otimes\Ff_p[N_1,\dots,N_{p-1}]^{S_p},\]
we get the desired isomorphism in the proposition.
\end{proof}

\begin{prop}\label{prop:image of Phi_1}
Let $R_0=\Ff_p[f,fh]\subset R$, and let $I=\Phi_1(H^*(BPU(p^2))\otimes\Ff_p)$. Then $R_0\otimes \wh L_p^{<p^3}\subset I$. Here $\wh L_p^i$ means the graded piece of $\wh L_p$ in polynomial degree $i$.
\end{prop}

The proof of the proposition will use the following two elementary lemmas. 
\begin{lem}\label{lem:f_2}
Let $f,h\in \Ff_p[\xi,\eta]$ be the elements as in \eqref{eq:Dickson}. Then we have
\[\xi^{p^2}\eta-\eta^{p^2}\xi=fh.\]	
\end{lem}
\begin{proof}
Since $f=\xi\eta(\xi^{p-1}-\eta^{p-1})$ and
\[h=\xi^{p^2-p}+\eta^{p-1}(\xi^{p-1}-\eta^{p-1})^{p-1}=\eta^{p^2-p}+\xi^{p-1}(\xi^{p-1}-\eta^{p-1})^{p-1},\]
it follows that $h\xi^p-f^{p-1}\xi=\xi^{p^2}$ and $h\eta^p-f^{p-1}\eta=\eta^{p^2}$.
Multiplying these two identities by $\eta$ and $\xi$, respectively, and then taking the difference gives the desired equation in the lemma.
\end{proof}

\begin{lem}\label{lem:kp^2}
For integers $i,j,k\geq 0$, the equation $kp^2=i(p+1)+j(p^2-p)$ holds if and only if
$k$ is even and 
\[i=\frac{p}{2}(k-r(p-1)),\ \ j=\frac{1}{2}(k+r(p+1)),\ \ -\lfloor\frac{k}{p+1}\rfloor\leq r\leq \lfloor\frac{k}{p-1}\rfloor.\]	
\end{lem}
\begin{proof}
Both $p+1$ and $p^2-p$ are even, so $k$ must be even. The mod $p$ reduction of the equation gives $p|i$, say $i=pq$. Dividing the equation by $p$, we obtain
\begin{equation}\label{eq:numerical equality}
	kp=q(p+1)+j(p-1).
\end{equation}
The mod $p$ reduction of \eqref{eq:numerical equality} gives $j\equiv q$ mod $p$.
Write $j=q+pr$ with $r\in\Zz$. Then, dividing \eqref{eq:numerical equality} by $p$ yields $k=2q+r(p-1)$, giving the formulas for $i,j$. The bounds on $r$ are exactly the conditions $i,j\geq 0$.	
\end{proof}

\begin{proof}[Proof of Proposition \ref{prop:image of Phi_1}]
Combining Lemma	\ref{lem:f_2} with \eqref{eq:Phi_m}, we immediately have $R_0\subset I$. Now we prove $\wh L_p^d\subset I$ ($d<p^3$) by induction on the polynomial degree $d$. The base case $d=0$ is trivial. 

Consider the commutative diagram
\begin{equation}\label{diag:L_p}
	\begin{gathered}
	\begin{tikzcd}
		H^*(BPU(p^2);\Ff_p)\arrow[r,"\bar\Delta^*"] \arrow[d,"\Phi_1"]
		& H^*(BPU(p);\Ff_p)\arrow[r] 
		& H^*(BT_{PU(p)};\Ff_p)^{S_p}\arrow[d,"\cong"]\\
		(A\otimes B)^{W_1}\cong R\otimes\wh L_p\arrow[r]
		& \wh L_p\arrow[hook,r]
		& L_p
	\end{tikzcd}
	\end{gathered}
\end{equation}
which is induced by the commutative diagram of group inclusions
\[\begin{tikzcd}
	T_{PU(p)} \arrow[r] \arrow[d]
	& \Gamma\times T_{PU(p)}\arrow[d]\\
	PU(p) \arrow[r]
	& PU(p^2) 
\end{tikzcd}\]

Note that the composition $T_{PU(p)}\hookrightarrow \Gamma\times T_{PU(p)}\hookrightarrow PU(p^2)$ is the same as the composition $T_{PU(p)}\xr{\bar\Delta} T_{PU(p^2)}\hookrightarrow PU(p^2)$, so Theorem \ref{thm:E_4} shows that the image of the composition of the upper horizontal maps in the diagram \eqref{diag:L_p} contains the image of the map $\Delta:\overline K_{p^2}\to \overline K_p$.
Hence, the commutativity of diagram \eqref{diag:L_p}, together with Proposition \ref{prop:L_p}, implies that for any $b\in \wh L_p$, there is an element
$\lambda\in H^*(BPU(p^2))\otimes\Ff_p$ such that 
\begin{equation}\label{eq:Phi_1}
\Phi_1(\lambda)=b+\sum_i a_i\otimes b_i,\ 
\text{ for some }a_i\in R^+,\ b_i\in\wh L_p.
\end{equation}

Now assume that $\wh L_p^{<d}\subset I$, $d<p^3$. For an arbitrary $b\in\wh L_p^d$, choose an element $\lambda\in H^{2d}(BPU(p^2))\otimes\Ff_p$ satisfying \eqref{eq:Phi_1}. Since $a_i\in R^+$, each $b_i\in\wh L_p^{<d}\subset I$. So it suffices to show that each $a_i\in R_0$. By the definition of $\wh L_p$, we may assume that the cohomology degree of $a_i$ is $2k_ip^2$ for some $0<k_i<p$. Write $a_i=r_i+abr_i'$ for some $r_i,r_i'\in \Ff_p[f,h]$. Since $\lambda\in H^{2d}(BPU(p^2))\otimes\Ff_p$, $\beta(\lambda)=0$. 
It follows that 
\[\beta\Phi_1(\lambda)=\sum_i\beta(a_i)\otimes b_i=s\sum_ir_i'\otimes b_i=0,\ \ s=\xi b-\eta a.\]
Thus $\sum_ir_i'\otimes b_i=0$, so we may assume that $a_i\in\Ff_p[f,h]$. 

When $k_i<p-1$, Lemma \ref{lem:kp^2} implies that $k_i$ is even and $a_i=f^{\frac{k_ip}{2}}h^{\frac{k_i}{2}}$ (changing the coefficient of $b_i$ if necessary), hence $a_i\in R_0$. When $k_i=p-1$, we have $b_i\in\wh L_p^0$ since $d<p^3$. So we may assume $\Phi_1(\lambda)=a+b$ for some $a\in \Ff_p[f,h]$. The following commutative diagram of group inclusions
\begin{equation}\label{diag:Gamma}
		\begin{gathered}
	\begin{tikzcd}
	\Gamma \arrow[rr, "i_1", "g\mapsto{(g,0)}"']\arrow[d]
	&&\Gamma\times\Gamma=E_2\arrow[d]\\
	E_1 \arrow[rr]
	& &PU(p^2) 
\end{tikzcd}
\end{gathered}
\end{equation}
implies that $a\in\im((Bi_1)^*\circ\Phi_2)$. However, in polynomial degrees $\leq p^3$, the polynomial part of $\im\Phi_2$ is generated by $f_1=\Phi_2(y_{p,0})$ and $f_2=\Phi_2(y_{p,1})$ by Theorem \ref{thm:symplectic invariants} and \eqref{eq:Phi_m}. Since $(Bi_1)^*(f_1)=f$ and $(Bi_1)^*(f_2)=fh$ by Lemma \ref{lem:f_2}, we still have $a\in R_0$, and the proof is complete. 
\end{proof}

\section{Proof of the main theorem}\label{sec:proof}
In this section, we assume that $p$ is an odd prime, and we identify  $H^*(BU(n))$ with $\Lambda_n$ via the correspondence $f(c_1,\ldots,c_n)\leftrightarrow f(\sigma_1,\ldots,\sigma_n)$.

Let $E_1,E_2\subset PU(p^2)$ be the two maximal nontoral elementary abelian $p$-subgroups of $PU(p^2)$ defined in Section \ref{subsec:p-subgroups}, and let $E_0\subset T_{PU(p^2)}$ be a maximal toral elementary abelian $p$-subgroup. From Theorem \ref{thm:Weyl group}, we know that \[\Phi=\Phi_0\oplus\Phi_1\oplus\Phi_2,\] 
where $\Phi_i$, $i=0,1,2$, are the restriction maps of mod $p$ cohomology:
\[
	\Phi_i:H^*(BPU(p^2);\Ff_p)\to H^*(BE_i;\Ff_p).
\]
Since $E_0\hookrightarrow PU(p^2)$ factors through $T_{PU(p^2)}$ and $H^*(BT_{PU(p^2)};\Ff_p)\to H^*(BE_0;\Ff_p)$ is injective, we may replace $E_0$ by $T_{PU(p^2)}$.

Let $\Phi_0'$ be the restriction map of integral cohomology
\[\Phi_0':H^*(BPU(p^2))\to H^*(BT_{PU(p^2)}).\]
By Theorem \ref{thm:E_4}, 
\[\im\Phi_0'=H^*(BT_{PU(p^2)})^W\cong K_{p^2}\subset\Lambda_{p^2},\]
hence there exists a class $\alpha\in H^*(BPU(p^2))$ such that $\Phi_0'(\alpha)=e$, where $e\in K_{p^2}$ is the element constructed in Section \ref{sec:element e}. 

Write $H^*(BE_2;\Ff_p)=P\otimes Q$, where 
\[P=\Ff_p[\xi_1,\eta_1,\xi_2,\eta_2],\ \ Q=\Lambda_{\Ff_p}[a_1,b_1,a_2,b_2]\]
are the polynomial and exterior subalgebras, respectively. 
\begin{prop}\label{prop:Phi_2(alpha)}
Let $i=\frac{(p-1)^2}{2}$ and $j=\frac{p-1}{2}$. Then there exists an integer $k$ such that $\Phi_2(\alpha)=\Phi_2(ky_{p,0}^iy_{p,1}^j)$.
\end{prop}
\begin{proof}
Since both $\alpha$ and the generators of $P$ have even cohomology degree, we can write
\[\Phi_2(\alpha)=\alpha_1+\alpha_2+\alpha_3,\]
where $\alpha_1\in P$, $\alpha_2\in P\otimes Q^2$, $\alpha_3\in P\otimes Q^4$. Furthermore, since $\alpha$ is an integral cohomology class, $\alpha$ maps to zero under the Bockstein. It follows that $\alpha_3=0$ since \[Q^4=\Ff_p\{a_1b_1a_2b_2\}\ \text{ and }\ \beta(a_1b_1a_2b_2)\neq 0.\] 
By Theorem \ref{thm:Weyl group}, $\Phi_2(\alpha)\in (P\otimes Q)^{Sp_4(\Ff_p)}$, and since $g:=-I_4\in Sp_4(\Ff_p)$ ($I_4$ is the identity matrix), it follows that
\[g\alpha_2=-\alpha_2=\alpha_2,\]
where the first equality comes from the fact that the polynomial part of $\alpha_2$ has odd polynomial degree $p^3-p^2-1$. Thus, $\alpha_2=0$, and so $\Phi_2(\alpha)\in P^{Sp_4(\Ff_p)}$. By Theorem \ref{thm:symplectic invariants} and Lemma \ref{lem:p^3-p^2} below, the only possible element of $P^{Sp_4(\Ff_p)}$ of the same degree as $\alpha$ is a multiple of $f_1^if_2^j$. Then, \eqref{eq:Phi_m} gives the desired result. 
\end{proof}
\begin{lem}\label{lem:p^3-p^2}
$p^3-p^2=i(p+1)+j(p^2+1)$, $i,j\in\Zz_{\geq0}$, holds if and only if
\[i=\frac{(p-1)^2}{2}, \ \ j=\frac{p-1}{2}.\]
\end{lem}
\begin{proof}
The `if' part can be verified directly. For the other direction, using $p\equiv -1$ mod $p+1$, we have
\[p^3-p^2\equiv -2,\ \ p^2+1\equiv 2\mod{p+1}.\]
It follows that $j=k\frac{p+1}{2}-1$ for some integer $k\geq 1$.
Substituting this into the original equation gives
\[i=\frac{p^3-p^2-j(p^2+1)}{p+1}=p^2-p+1-\frac{k(p^2+1)}{2}.\]
If $k\geq 2$, then
$i\leq p^2-p+1-(p^2+1)=-p<0$,
contradicting the assumption $i\geq 0$. Thus $k=1$, and the equations for $i,j$ follow.
\end{proof}

Note that $\Phi_0(y_{p,0})=\Phi_0(y_{p,1})=0$. Thus Proposition \ref{prop:Phi_2(alpha)} shows that by rechoosing $\alpha$ if necessary, we have $\Phi_2(\alpha)=0$. We shall assume this in what follows.

\begin{prop}\label{prop:Phi_1(alpha)}
There exists an integral class $\alpha_0$ in the ideal 
\[\langle y_{p,0},y_{p,1}\rangle\subset H^*(BPU(p^2)),\] 
such that $\Phi_1(\alpha)=\Phi_1(\alpha_0)$ and $\Phi_2(\alpha_0)=0$.	
\end{prop}
\begin{proof}
Adopting the notation of Section \ref{sec:image of Phi_1}, we have \[\im\Phi_1\subset(A\otimes B)^{W_1}\cong R\otimes \wh L_p\]
by \eqref{eq:image of Phi_1} and Proposition \ref{prop:the W_1 invariant ring}.	In particular, one can show that $\Phi_1(\alpha)\in\Ff_p[f,h]\otimes\wh L_p$, using the fact that $\beta(\alpha)=0$ and arguing as in the proof of Proposition \ref{prop:Phi_2(alpha)}.
Write
\begin{equation}\label{eq:Phi_1(alpha)}
\Phi_1(\alpha)=\sum_i a_i\otimes b_i,\ \ a_i\in \Ff_p[f,h],\ \ b_i\in\wh L_p.
\end{equation}
The polynomial degree of each $b_i$ is $k_ip^2$ for some integer $0\leq k_i\leq p-1$, so $b_i\in\im\Phi_1$ by Proposition \ref{prop:image of Phi_1}.

Let $\Psi$ be the projection map $R\otimes\wh L_p\to\wh L_p$. The commutative diagram of group inclusions
\[\begin{tikzcd}
	T_{PU(p)} \arrow[r] \arrow[d]
	& T_{PU(p^2)}\arrow[d]\\
	\Gamma\times T_{PU(p)}\arrow[r]
	& PU(p^2) 
\end{tikzcd}\]
implies that $\Delta\Phi_0=\Psi\Phi_1$. Since $\Delta\Phi_0'(\alpha)=\Delta(e)\equiv0$ mod $p$ by Proposition \ref{prop:e}, it follows that $\deg(a_i)>0$ for every $a_i$ in \eqref{eq:Phi_1(alpha)}. Hence, $k_i<p-1$ for every $i$. On the other hand, let $\Psi'$ be the projection $R\otimes\wh L_p\to R$. The diagram \eqref{diag:Gamma} shows that  $\Psi'\Phi_1=(Bi_1)^*\Phi_2$. Since $\Phi_2(\alpha)=0$, we may assume $k_i>0$ for all $i$.

Since $a_i\in \Ff_p[f,h]$, applying Lemma \ref{lem:kp^2} to the polynomial degrees of $f,h$ and $a_i$ shows that $k_i$ is even and $a_i=f^rh^s$ (changing the coefficient of $b_i$ if necessary) with 
\[(r,s)=(\frac{p(p-1-k_i)}{2},\frac{p-1-k_i}{2})\ \text{ or } \ (r,s)=(0,p).\]
The second case can be ruled out because it forces $b_i\in\wh L_p^0$ for degree reasons. 
In the first case, since $k_i<p-1$, $r>s>0$. So $a_i=f^{r-s}(fh)^s$. By \eqref{eq:Phi_m} and Lemma \ref{lem:f_2}, we know that $\Phi_1(y_{p,0})=f$ and $\Phi_1(y_{p,1})=fh$. From here, the existence of $\alpha_0\in\langle y_{p,0},y_{p,1}\rangle$ such that $\Phi_1(\alpha)=\Phi_1(\alpha_0)$ follows easily.

It remains to show that $\Phi_2(\alpha_0)=0$. This is a consequence of the following facts: $\Psi'\Phi_1(\alpha_0)=(Bi_1)^*\Phi_2(\alpha_0)=0$; $\Phi_2(\alpha_0)\in\Ff_p[f_1,f_2]$ by the argument used in the proof of Proposition \ref{prop:Phi_2(alpha)}; $(Bi_1)^*:\Ff_p[f_1,f_2]\cong \Ff_p[f,fh]$ as $(Bi_1)^*(f_1)=f$, $(Bi_1)^*(f_2)=fh$ by Lemma \ref{lem:f_2}.
\end{proof}
Proposition \ref{prop:Phi_1(alpha)} shows that by rechoosing $\alpha$ if necessary, we also have $\Phi_1(\alpha)=0$. Furthermore, since  $|\chi|=3$, $\Phi_0(\chi\alpha)=0$. So Theorem \ref{thm:main} is a consequence of the following proposition.

\begin{prop}\label{prop:BPU9 mod 3}
$\chi\alpha\neq 0$ in $H^*(BPU(p^2);\Ff_p)$. 
\end{prop}
\begin{proof}
We need to show that $\chi\alpha\notin pH^*(BPU(p^2))$. Consider the integral Serre spectral sequence $E_*^{*,*}$ associated to \eqref{eq:fibration} for $BPU(p^2)$. By construction, the image of $\alpha$ in $ E_\infty^{0,*}\cong K_{p^2}$ is $e$, hence the representative of $\chi\alpha$ in $E_\infty^{3,*}$ is $xe$. Then Proposition \ref{prop:e not in} implies that $\chi\alpha\neq 0$ in $E_\infty^{3,*}$ since $d_3=\nabla$ is the only possible differential mapping to $E_{i}^{3,*}$ ($i\geq 2$). 
	
	Suppose on the contrary that there exists an element $\lambda\in H^*(BPU(p^2))$ such that $\chi\alpha=p\lambda$. 
	Then, $\lambda\neq 0$ in the first nontrivial odd column $E_\infty^{3,*}$ since $p\lambda=\chi\alpha\neq 0$ in $E_\infty^{3,*}$.
	Suppose that $xe'$ ($e'\in\Lambda_{p^2}$) is a representative of $\lambda$ in \[E_\infty^{3,*}\subset x\left(\frac{\Lambda_{p^2}}{\nabla(\Lambda_{p^2})}\right).\] 
	Then there exists $c\in\Lambda_{p^2}$ such that 
	$e=pe'+\nabla(c)$.
	Thus, $\Delta(e)=p\Delta(e')+\nabla\Delta(c)$ for the map $\Delta:\Lambda_{p^2}\to\Lambda_p$. 
	Since $p\nmid\deg(c)$, $\Delta(c)\equiv 0$ mod $p$ by \eqref{eq:diagonal map}. On the other hand, by Corollary \ref{cor:Delta(e)} we have 
	\[\Delta(e)=p(k\delta^p+\nabla(a))\text{ for some }a\in\Lambda_p,\ p\nmid k.\] 
   Hence, 
  \[p\Delta(e')=pk\delta^p+p\nabla(a')\text{ for some }a'\in\Lambda_p,\]
   and then $\Delta(e')=k\delta^p+\nabla(a')$ in $\Lambda_p$.
  By Proposition \ref{prop:the map phi}, this implies that 
  \begin{equation}\label{eq:lambda}
  	\bar\Delta^*(\lambda)=k\chi\delta^p+\rho\text{ for some }\rho\in\langle\chi\gamma\rangle\subset H^*(BPU(p)),
  \end{equation} 
  where $\bar\Delta^*:H^*(BPU(p^2))\to H^*(BPU(p))$ is the homomorphism induced by the map $\bar\Delta:BPU(p)\to BPU(p^2)$ defined in \eqref{eq:bar Delta}.
  
  Let $\Gamma$ be the unique (up to conjugation) maximal nontoral elementary abelian $p$-subgroup of $PU(p)$, and let $i:\Gamma\to PU(p)$ be the inclusion.
  Recall from \eqref{eq:f,h} and the definitions of $s,f,h\in H^*(B\Gamma;\Ff_p)$ that 
  \begin{gather*}
  	(Bi)^*(\chi)=s,\ \ (Bi)^*(\gamma)=f,\ \ (Bi)^*(\delta)=-h,\\ 
  	\beta(s)=\beta(f)=\beta(h)=0,\ \ \beta P^1(s)=f,\ \ P^1(f),\,P^1(h)\in\Ff_p[f,h].
\end{gather*}
  Hence, using \eqref{eq:lambda} and Corollary \ref{cor:cohomology of Gamma}, we have
  \[\beta P^1(Bi)^*\bar\Delta^*(\lambda)=-kfh^p+b\text{ for some }b\in\langle f^2\rangle\subset\Ff_p[f,h],\] 
which is a nonzero element of the subring $\Ff_p[f,h]\subset H^*(B\Gamma;\Ff_p)$ since $p\nmid k$.
 
From the commutative diagram of group inclusions
\[\begin{tikzcd}
	\Gamma \arrow[rr, "i_1", "g\mapsto{(g,0)}"']\arrow[d,"i"]
	&&\Gamma\times\Gamma=E_2\arrow[d]\\
	PU(p) \arrow[rr]
	& &PU(p^2) 
\end{tikzcd}\]
we see that $(Bi)^*\bar\Delta^*=(Bi_1)^*\Phi_2$. Hence there is an element $\lambda'\in P\subset H^*(BE_2;\Ff_p)$ ($P$ is the polynomial subring of $H^*(BE_2;\Ff_p)$) such that 
  \[(Bi_1)^*(\lambda')=-kfh^p+b.\]
  It is clear that $\lambda'\in P^{W_2}$ for the Weyl group $W_2=W_{PU(p^2)}(E_2)=Sp_{4}(\Ff_p)$.
  Then by Theorem \ref{thm:symplectic invariants}, $\lambda'\in\langle f_1,f_2\rangle^2\subset\Ff_p[f_1,f_2]$ for degree reasons.
  However, we have
  $(Bi_1)^*(f_1)=f$, and $(Bi_1)^*(f_2)=fh$ by Lemma \ref{lem:f_2}.
   So  
   \[(Bi_1)^*(\lambda')\in\langle f^2\rangle\subset\Ff_p[f,h].\] 
   Since $k\neq 0$, $-kfh^p+b\notin\langle f^2\rangle$, a contradiction.
\end{proof} 

\section{An abelian subgroup  detecting $\chi\alpha$}\label{sec:detect}
In this section, we construct a subgroup $\Gamma(n)\subset PU(n)$ isomorphic to $(\Zz/n)^2$. When $n=p^2$ for $p$ an odd prime, we show that this subgroup detects the integral cohomology class $\chi\alpha\in H^*(BPU(p^2))$.

This subgroup is a generalization of the nontoral elementary abelian $p$-subgroup $\Gamma$ of $PU(p)$ defined in Section \ref{subsec:PUp}. It is generated by the quotient images of the two elements of $U(n)$:
\[\tilde\sigma=\begin{pmatrix}
	0& 1\\
	I_{n-1} &0
\end{pmatrix},\quad
\tilde\tau=(\omega,\omega^2,\dots,\omega^{n-1},1)\in T^n,\]
where $\omega=e^{2\pi i/n}$.
Denote this subgroup of $PU(n)$ by $\Gamma(n)$ and let $\sigma$, $\tau$ denote the respective images of $\tilde\sigma$ and $\tilde\tau$ in $PU(n)$. We have $\tilde\tau\tilde\sigma=\omega\tilde\sigma\tilde\tau$, so that $\Gamma(n)\cong(\Zz/n)^2$. Let $i_n:\Gamma(n)\hookrightarrow PU(n)$ be the inclusion map. 

Suppose that $p>2$ is a prime factor of $n$ and that $q=p^r$ is the $p$-primary factor of $n$. Then the mod $q$ cohomology ring of $B\Gamma(n)$ is given by
\[H^*(B\Gamma(n);\Zz/q)\cong\Zz/q[\xi,\eta]\otimes\Lambda_{\Zz/q}[a,b],\]
where $a,b\in H^1(B\Gamma(n);\Zz/q)\cong \Hom(\Gamma(n),\Zz/q)$ are defined by $a(\sigma)=b(\tau)=1$, $a(\tau)=b(\sigma)=0$, and $\xi=\beta_r(a)$, $\eta=\beta_r(b)$. Here $\beta_r$ is the Bockstein homomorphism associated to the short exact sequence
\[0\to\Zz/q\xr{\cdot q}\Zz/q^2\to\Zz/q\to0.\]
\begin{lem}\label{lem:Gamma(n)}
The image of the generator $\chi$ of $H^3(BPU(n);\Zz/q)\cong \Zz/q$ under the restriction map $(Bi_n)^*:H^*(BPU(n);\Zz/q)\to H^*(B\Gamma(n);\Zz/q)$ is  $k(\xi b-\eta a)$ for some $k\in\Zz/q$ with $k\not\equiv 0$ mod $p$.	
\end{lem}

\begin{proof}
	 Let $C\cong\Zz/n$ be the center of $SU(n)$, which is generated by the diagonal matrix $\omega I_n$, and let $c$ be a generator of $H^1(BC;\Zz/q)\cong \Zz/q$. Note that $PU(n)=SU(n)/C$. Let $G\subset SU(n)$ be the subgroup such that there is a group extension
	 \[1\to C\to G\to \Gamma(n)\to 1.\]
	 First we will prove that $d_2(c)=kab$ for some $k\in\Zz/q$ with $k\not\equiv 0$ mod $p$ on the $E_2$-page of the $\Zz/q$-cohomology Serre spectral sequence of the fibration
	 \[BC\to BG\to B\Gamma(n).\]
	 
	 A standard result from group cohomology theory tells us that 
	 \[H^1(BG;\Zz/q)\cong \Hom(G/[G,G],\Zz/q)\cong (\Zz/q)^2,\]
	 where the second isomorphism comes from the fact that $[G,G]=C$. So the kernel of the map $d_2:E_2^{0,1}\to E_2^{2,0}$ is zero, hence $d_2(c)=i\xi+j\eta+kab$ for some $i,j,k\in\Zz/q$ with at least one of $i,j,k$ nonzero modulo $p$. 
	 
	 Let $G_1,G_2$ be the subgroups of $G$ generated by $\varepsilon\tilde\sigma$ and $\varepsilon\tilde\tau$, respectively, where $\varepsilon=1$ for $n$ odd and $e^{\pi i/n}$ for $n$ even. Note that $\det(\tilde\sigma)=\det(\tilde\tau)=(-1)^{n-1}$. Write 
	 \[H^*(BG_1;\Zz/q)\cong\Zz/q[\xi]\otimes\Lambda_{\Zz/q}[a],\ \ H^*(BG_2;\Zz/q)\cong\Zz/q[\eta]\otimes\Lambda_{\Zz/q}[b].\] 
	 Then the following compositions 
	 \[\pi_1:BG_1\to BG\to B\Gamma(n),\ \ \pi_2:BG_2\to BG\to B\Gamma(n)\]
	 induce surjections $\pi_1^*$ and $\pi_2^*$ on $\Zz/q$-cohomology and satisfy 
\[\begin{split}
	&\pi_1^*(\xi)=\xi,\ \ \pi_1^*(\eta)=\pi_1^*(ab)=0,\\
	&\pi_2^*(\eta)=\eta,\ \ \pi_2^*(\xi)=\pi_2^*(ab)=0.
\end{split}\]
	  This implies that $i\equiv j\equiv0$ mod $q$ in the above formula for $d_2(c)$, and the statement for $d_2(c)$ follows.
	 
	 Now to prove the lemma, consider the commutative diagram 
	 \[
	 \xymatrix{
	 	BC\ar[r]\ar[d]^{=}&BG\ar[r]\ar[d]&B\Gamma(n)\ar[d]^{Bi_n}\\
	 	BC\ar[r]&BSU(n)\ar[r]&BPU(n)}
	 \]
	 Using the spectral sequence map, we see that $(Bi_n)^*(y)=ab$ for some generator $y$ of $H^2(BPU(n);\Zz/q)\cong\Zz/q$. Since $\beta_r(y)=k\chi$ for some $k\in\Zz/q$ with $k\not\equiv 0$ mod $p$ and $\beta_r(ab)=\xi b-\eta a$, the lemma follows immediately.	 
\end{proof}
\begin{prop}
Let $\alpha\in H^*(BPU(p^2))$ be the class constructed in Section \ref{sec:proof}. Then $(Bi_{p^2})^*(\chi\alpha)\neq 0$ in $H^*(B\Gamma(p^2))$.
\end{prop}
\begin{proof}
Let $G_2\cong\Zz/p^2$ be the toral subgroup of $U(p^2)$ generated by $\tilde\tau$ and let $j:G_2\hookrightarrow T^{p^2}$ be the inclusion map. Then the induced map $(Bj)^*$ on integral cohomology can be identified with the ring homomorphism  
\[\Zz[v_1,\dots,v_{p^2}]\to\Zz[\eta]/\langle p^2\eta\rangle,\ \ v_i\mapsto i\eta.\]
Hence, the image of $\alpha$ under the composition of $H^*(BPU(p^2))\to H^*(BU(p^2))$ and $H^*(BU(p^2))\to H^*(BG_2)$ is nonzero by the formula for $\Theta(e)$ from Proposition \ref{prop:e}. This implies that 
\[(Bi_{p^2})^*(\alpha)\not\equiv 0\bmod I,\]
where $I$ is the kernel of the restriction map $H^*(B\Gamma(p^2))\to H^*(BG_2)$.

Note that 
\[H^*(B\Gamma(p^2))\cong\frac{\Zz[\xi,\eta,s]}{\langle p^2\xi,p^2\eta,p^2s,s^2\rangle},\]
where the $\Zz/p^2$ reduction of $s$ is $\xi b-\eta a$. So the $\Zz/p^2$ reduction of $(Bi_{p^2})^*(\alpha)$ takes the form $k\eta^{p^3-p^2}+b$ for some $0\neq k\in\Zz/p^2$ and $b\in\langle\xi\rangle\subset\Zz/p^2[\xi,\eta]$. Combining this with Lemma \ref{lem:Gamma(n)}, we get the desired result for $(Bi_{p^2})^*(\chi\alpha)$.
\end{proof}

\section*{Acknowledgment}
The author would like to thank Antonio Viruel for a conversation about the history of Adams' conjecture. The author also thanks the reviewer for pointing out an error in an earlier version of this paper.

\bibliography{M-A}
\bibliographystyle{amsplain}
\end{document}